\documentclass[
final
]{dmtcs-episciences}

\usepackage[utf8]{inputenc}
\usepackage{subfigure}

\usepackage[round]{natbib}

\usepackage{amsmath}
\usepackage{amssymb}
\usepackage{amsthm}
\usepackage{multicol}
\usepackage{comment}
\usepackage{xcolor}
\usepackage{thmtools}
\usepackage{tabularx}

\newcommand{\calL}{\mathop{\mathcal{L}}}
\newcommand{\calG}{\mathop{\mathcal{G}}}

\newcommand{\calM}{\mathop{\mathcal{M}}}
\newcommand{\calS}{\mathop{\mathcal{S}}}
\newcommand{\calC}{\mathop{\mathcal{C}}}

\theoremstyle{plain}
\newtheorem{thm}{Theorem}[section]
\newtheorem{lem}[thm]{Lemma}
\newtheorem{cor}[thm]{Corollary}

\newtheorem{prop}[thm]{Proposition}

\newtheorem{obs}[thm]{Observation}

\theoremstyle{definition}
\newtheorem{defn}[thm]{Definition}
\newtheorem*{convention}{Convention}
\newtheorem*{notation}{Notation}
\newtheorem{exam}[thm]{Example}
\newtheorem{hypothesis}{Hypothesis}
\newtheorem{question}{Question}
\newtheorem{algorithm}[thm]{Algorithm}

\theoremstyle{remark}
\newtheorem*{rem}{Remark}

\newtheorem*{note}{Note}

\author[K. Adaricheva et al.]{K. Adaricheva\affiliationmark{1}
  \and A. Mata\affiliationmark{2}
  \and S. Silberger\affiliationmark{1}
  \and A. Zamojska-Dzienio\affiliationmark{2}}

\title[Maximal Sublattices of Finite SD-lattices]{Conjecture on Maximal
  Sublattices of Finite Semidistributive Lattices and Beyond}

\affiliation{
  Department of Mathematics, Hofstra University, Hempstead NY, USA\\
  Faculty of Mathematics and Information Science, Warsaw University of
  Technology, Warsaw, Poland}


\begin{document}

\publicationdata{vol. 28:2}{2026}{32}{10.46298/dmtcs.16374}{2025-08-20; 2025-08-20; 2026-05-13}{2026-05-19}

\maketitle

\begin{center}
\emph{Dedicated to the memory of Donald Morison Silberger,
February 26, 1930 -- July 22, 2025.}
\end{center}

\noindent\rule{\linewidth}{0.4pt}

\begin{abstract}
We study maximal sublattices of finite semidistributive lattices via
their complements. We focus on the conjecture that such complements
are always intervals, which is known to be true for bounded lattices.
Since the class of semidistributive lattices is the intersection of
classes of join- and meet-semidistributive lattices, we study also
complements for these classes, and in particular convex geometries of
convex dimension 2, which is a subclass of join-semidistributive
lattices. In the latter case, we describe the complements of maximal
sublattices completely, as well as the procedure of finding all
complements of maximal sublattices.
\end{abstract}

\section{Introduction}\label{S:Intro}
\subsection{Background and motivation}

 A proper sublattice \begin{math}\calS\end{math} of a lattice \begin{math}\calL\end{math} is said to be {\em maximal} if for every \begin{math}x\in\calL\setminus\calS\end{math}, the sublattice generated by \begin{math}\calS\cup \{ x\}\end{math} is all of \begin{math}\calL\end{math}. Maximal sublattices of finite distributive lattices were fully described in  \cite[Theorem 4]{Chen73} and \cite[Theorem 3]{R73}:  the complements of maximal sublattices of distributive lattices are always intervals of the form \begin{math}\left[a, b\right] = \{x \in L\ : \ a \leq x \leq b\}\end{math} with \begin{math}a\end{math} being a unique join-irreducible element, and \begin{math}b\end{math} being a unique meet-irreducible one in the interval.

In \cite[Lemma 1]{R73} it was first shown that the complement of a proper \begin{math}(0,1)-\end{math} sublattice of a lattice of finite length (in particular, for a finite one) contains an interval \begin{math}\left[a, b\right]\end{math} with \begin{math}a\end{math} being a join-irreducible element and \begin{math}b\end{math} being a meet-irreducible one. We are interested in the complements of maximal sublattices (which are in relevant  cases \begin{math}(0,1)-\end{math} sublattices).

In this paper we focus on finite lattices and for them the interrelationships between classes we are interested in are shown in Figure \ref{fig:classes} (for details see \cite{Gan19} and \cite[Section 1.4]{S99}).

\begin{figure}[hbt]
  \begin{center}
    \includegraphics[width=0.5\linewidth]{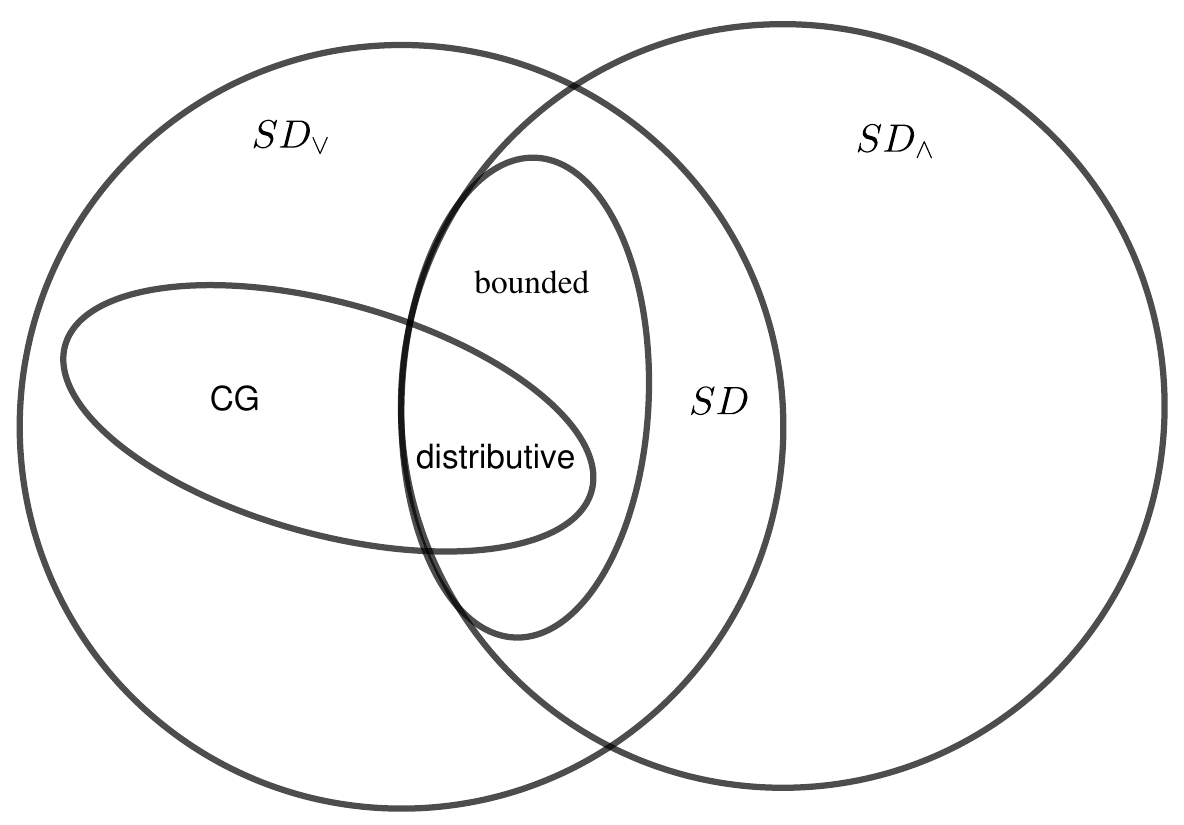}
    \caption{The inclusion relations between classes of lattices we focus on in this paper.}\label{fig:classes}
  \end{center}
\end{figure}

The class of distributive lattices is contained in the class of bounded lattices; by which we mean the bounded homomorphic images  of free lattices (for the precise definition see \cite[Chapter II]{FJN95} ).
The complements of maximal sublattices of bounded lattices were also shown to be intervals in  \cite[Theorem 7]{AFNS97}.

This suggests the following question.
\begin{question}\label{ques1}
For which classes of lattices are the complements of maximal sublattices intervals?
\end{question}

The class of semidistributive lattices is the extension of the class of bounded lattices, and
in \cite[page 116]{AFNS97}
the same question was asked of SD lattices.

In this paper we start a systematic study of this hypothesis:

\begin{hypothesis}\label{hyp1}
The complements of maximal sublattices of (finite) SD lattices are also intervals.
\end{hypothesis}

First, we will briefly recall some basic notions necessary to understand the formulation of problems which we investigate here.

A join-semidistributive (SD\begin{math}_{\vee}\end{math}) lattice is a lattice \begin{math}\left< \calL , \vee , \wedge\right>\end{math} in which the following quasi-identity holds:
\begin{displaymath}
x\vee y=x\vee z \implies x\vee (y\wedge z) =x\vee y
\end{displaymath}
for all \begin{math}x,y,z\in\calL\end{math}. Meet-semidistributive lattices (SD\begin{math}_\wedge\end{math}) are dually defined.

A lattice is said to be {\em semidistributive} (SD) if it is both join- and meet-semidistributive.

The  examples of \emph{join-semidistributive} lattices shown in Figure \ref{fig:maxnumbermaxelements}  demonstrate that complements of maximal sublattices  may have an arbitrary number of maximal elements.   All these examples come from a special subclass of join-semidsitributive lattices known as \emph{convex geometries} (CGs).

\begin{figure}[hbt]
  \begin{center}
    \includegraphics[width=0.7\linewidth]{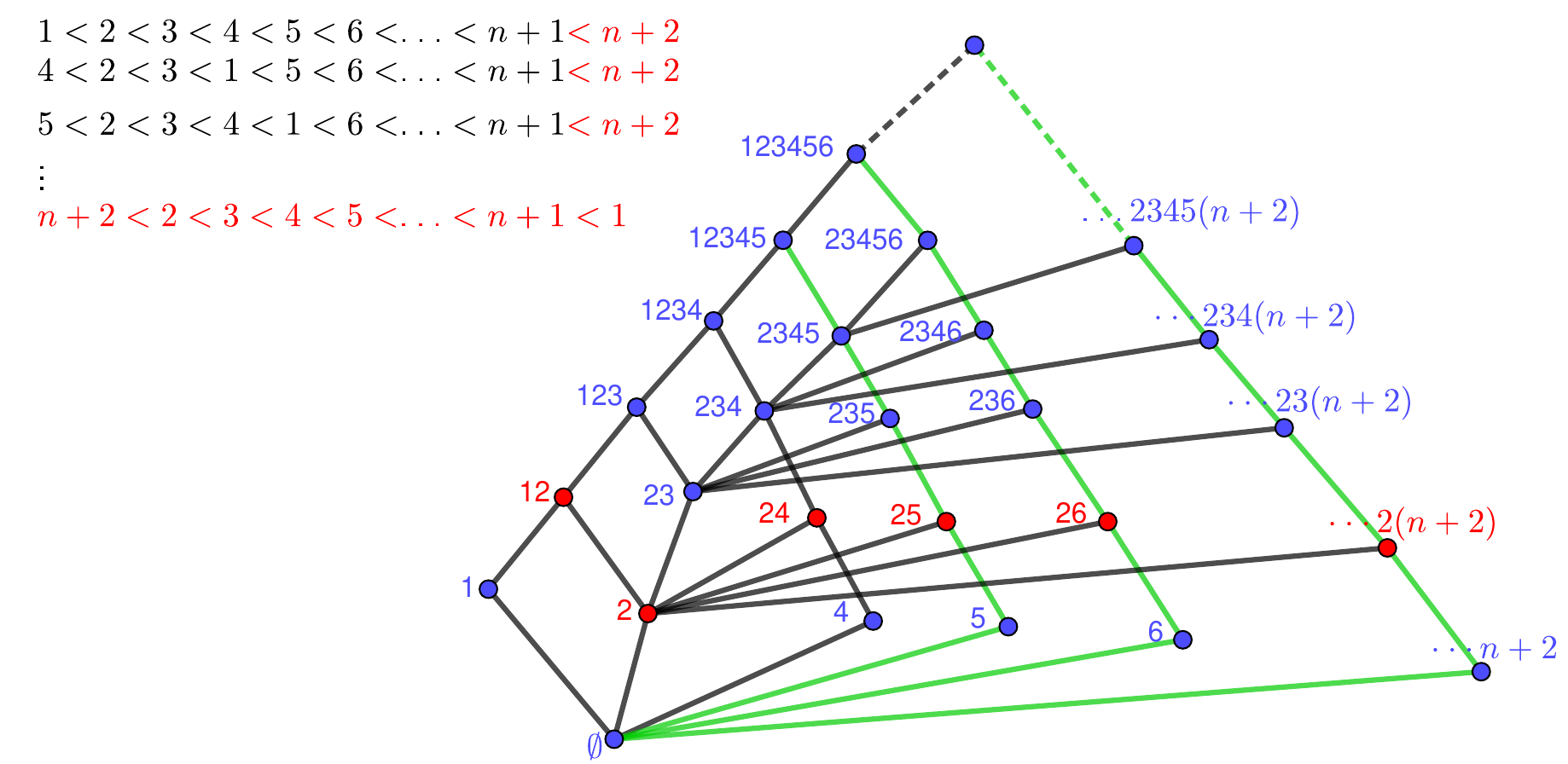}
     \caption{The figure above shows that in a CG with \begin{math}cdim=n\end{math}, which will also be SD\begin{math}_\vee\end{math}, there can be a complement of a maximal sublattice with \begin{math}n\end{math} distinct maximal elements. Here \begin{math}\left\{ 2, 12, 24, 25, 26, \ldots , 2(n+2)\right\} \end{math} is the complement of a maximal sublattice.}
    \label{fig:maxnumbermaxelements}
  \end{center}
\end{figure}

A {\em convex geometry} (CG) \begin{math}\left< \calG, \vee, \wedge\right>\end{math} is an  SD\begin{math}_\vee\end{math} lattice which is also  {\em lower semi-modular}:
\begin{equation*}
\forall\ x,y \in \calG\left( x \prec  x \vee y  \implies x \wedge y \prec y\right),
\end{equation*}
where \begin{math}a\prec b\end{math} means \begin{math}a\end{math} is a subcover of \begin{math}b\end{math}; that is, \begin{math}a<b\end{math} and there is no \begin{math}c\in\calG\end{math} for which \begin{math}a<c<b\end{math}.
Due to results of \cite{EJ85} a CG can be represented via intersections of down-sets of chains formed on the base set of a geometry, and hence the examples can be produced easily, unlike semidistributive lattices. Our main result, Theorem \ref{BigCGThm}, states that the complements of maximal sublattices of convex geometries generated by just two chains is either an interval or a union of two intervals with the common minimal element. In examples in Figure \ref{fig:maxnumbermaxelements},  the complements of maximal sublattices may include several intervals sharing the same minimal element.

This motivates the following hypothesis for the class of convex geometries and, more generally, join-semidistributive lattices.

\begin{hypothesis}\label{hyp2}
The complement of a maximal sublattice of a (finite) CG, or, more generally, lattices satisfying SD\begin{math}_{\vee}\end{math}, is the union of intervals with a common minimal element.
\end{hypothesis}

Immediately, one has also the dual formulation that the complement of a maximal sublattice of a (finite) SD\begin{math}_\wedge\end{math} lattice is the union of intervals with a common maximal element.

 If the Hypothesis \ref{hyp2} is true, then Hypothesis \ref{hyp1} follows.

If we do not restrict ourselves to SD\begin{math}_\vee\end{math} or SD\begin{math}_\wedge\end{math} lattices, we can have both more than one minimal and more than one maximal element in the complement of a maximal sublattice as Figure \ref{fig:fig10} shows.

\begin{figure}[hbt]
  \begin{center}
    \includegraphics[width=0.25\linewidth]{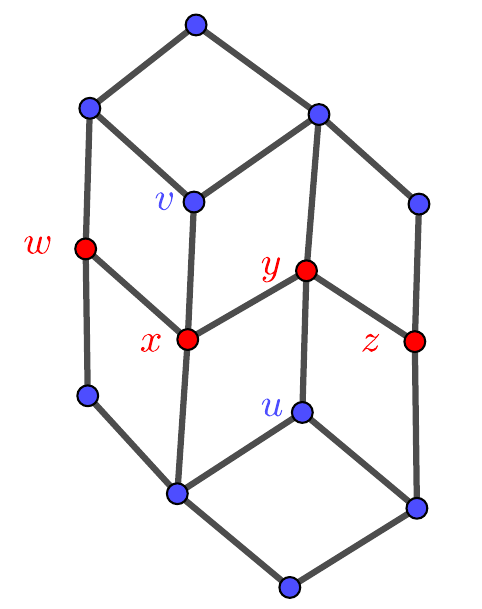}
     \caption{This lattice is an example of a lattice that is neither SD\begin{math}_\vee\end{math} nor SD\begin{math}_\wedge\end{math} that has a complement of a maximal sublattice with more than one minimum and more than one maximum. The set \begin{math}\{ w,x,y,z\}\end{math} is the complement of a maximal sublattice.  It is not SD\begin{math}_\vee\end{math} because \begin{math}y=u\vee x=u\vee z\end{math}, but \begin{math}u\vee (x\wedge z)=u\vee 0=u\neq u\vee x\end{math}. Similarly,  it is not SD\begin{math}_\wedge\end{math} because \begin{math}x=v\wedge w=v\wedge y\end{math}, but \begin{math}v\wedge (w\vee y)=v\wedge 1=v\neq v\wedge w\end{math}. }
    \label{fig:fig10}
  \end{center}
\end{figure}

CGs provide a combinatorial abstraction of convex sets in geometry, and their duals are known as anti-matroids \cite{AN16}. The antimatroids of convex geometries generated by 2 chains are also known as
\emph{slim planar semimodular} (SPS) lattices \cite{CK19}.

As we can see on Figure \ref{fig:classes}, the class of SD lattices contains the class of bounded lattices, which in turn contains the class of distributive lattices.   All CGs  which are also SD\begin{math}_\wedge\end{math}  are distributive, see, for example, \cite[Implication N98]{Gan19}, which is also an easy exercise.

Thus, CGs provide a very different generalization of distributivity compared to bounded lattices.
Here are some statistics on the number of convex geometries on 4 and 5 element due to \cite{5circles}: 15 distributive geometries out of 34 (44\%) and 49 distributive geometries among 672, thus, only about 7\%. Given these two observations, we would estimate a decreasing portion of distributive geometries among all CGs, when the number of elements in the base set increases.

Since the complements of the maximal sublattices in bounded lattices are intervals, they are, in particular, order convex: if \begin{math}a\leq d\leq b\end{math} and \begin{math}a,b \in \calC\end{math}, then \begin{math}d \in \calC\end{math}. Observation \ref{1maxORmin} (\ref{part36a}) and (\ref{part36b}) from Section \ref{observations} says that in certain situations the complement of a maximal sublattice of an arbitrary lattice is order convex.

Moreover, all of our examples of complements of maximal sublattices have been order convex. All these together suggest the following conjecture:

\begin{hypothesis}\label{hyp3}
In any finite lattice \begin{math}\calL\end{math} the complement \begin{math}\calC\end{math} of any maximal sublattice \begin{math}\calM\end{math} is order convex.
\end{hypothesis}

As noted in Observation \ref{maxmeetirreducible} below, any minimal element of the complement of a \begin{math}(0,1)-\end{math}sublattice of any finite lattice \begin{math}\calL\end{math} must be join-irreducible and any maximal element must be meet-irreducible. Hence the complement of any \begin{math}(0,1)-\end{math}sublattice of a given lattice must have at least one join-irreducible element, each minimum element, and at least one meet-irreducible element, each maximal element.

Recall that in a finite distributive lattice, every complement of a maximal sublattice is of the form \begin{math}[a,b]\end{math} where \begin{math}a\end{math} is the only join-irreducible element and \begin{math}b\end{math} is the only meet-irreducible one in the interval \begin{math}[a,b]\end{math} (\cite[Theorem 4]{Chen73} and \cite[Theorem 3]{R73}). The example in Figure \ref{daydoubling2}, using Alan Day's doubling construction to produce a bounded lattice from a distributive one \cite{Day1992} shows that this is not necessarily the case with bounded lattices. Figure \ref{fig:JNExample} shows another example, this one SD but not bounded.   However, none of our examples of complements of maximal sublattices of convex geometries have more than one join-irreducible, the unique minimum, and their only meet-irreducible elements are the maxima. This suggests the following question.

\begin{figure}[hbt]
  \begin{center}
    \includegraphics[width=0.5\linewidth]{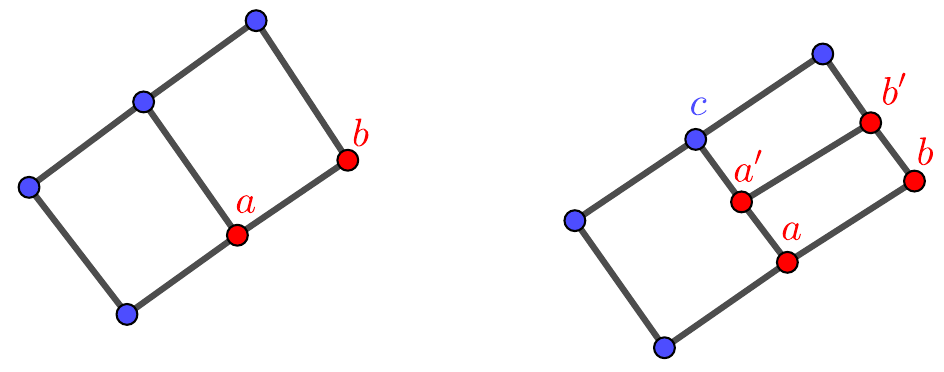}
     \caption{The lattice \begin{math}\calL\end{math} on the left is a distributive lattice and the lattice \begin{math}\calL'\end{math} to the right is an example of a lattice derived by doubling the interval \begin{math}[a,b]\end{math}. Day's 1992 paper \cite{Day1992} shows that bounded lattices are precisely those obtained by a sequence of interval doublings of this nature starting with a distributive lattice. Thus, \begin{math}\calL'\end{math} is bounded, but you can check that the interval \begin{math}[a,c]\end{math} is the complement of a maximal sublattice and both \begin{math}a\end{math} and \begin{math}a'\end{math} are join-irreducible.}\label{daydoubling2}
  \end{center}
\end{figure}

\begin{figure}[hbt]
  \begin{center}
    \includegraphics[width=0.5\linewidth]{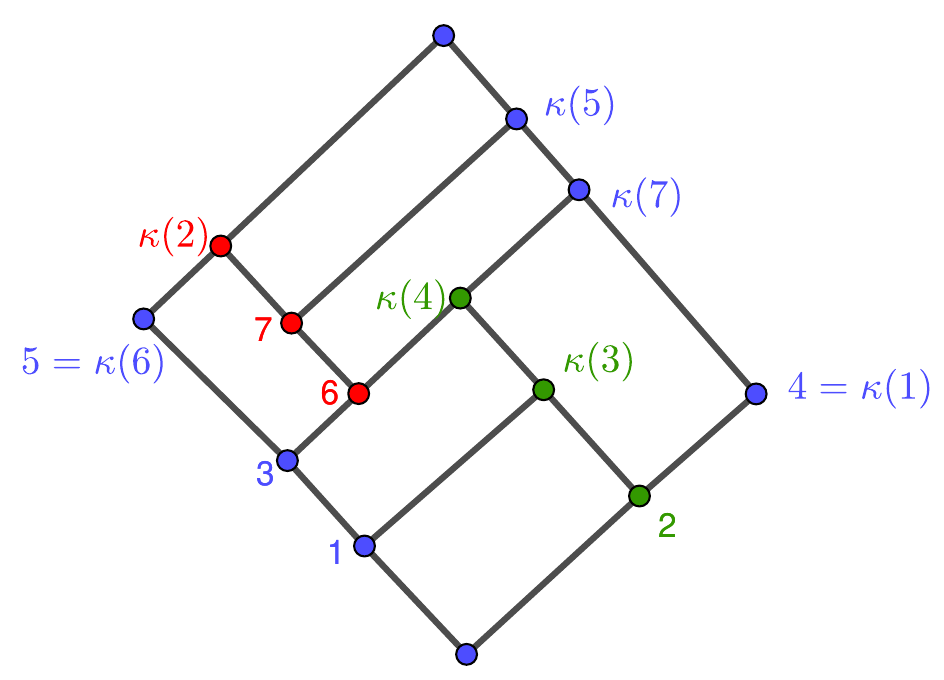}
     \caption{This lattice is an example, due to J\'onsson and Nation \cite[Figure 5.5]{FJN95} of an SD lattice that is not bounded. The elements of \begin{math}Ji(\calL)\end{math} are labeled with numbers and those in \begin{math}Mi(\calL)\end{math} are labeled as the image of the function \begin{math}\kappa\end{math}, Theorem \ref{Def12Thm}, showing that \begin{math}\kappa\end{math} is a bijection from \begin{math}Ji(\calL)\end{math} onto \begin{math}Mi(\calL)\end{math} and hence \begin{math}\calL\end{math} is SD. One can check that the interval \begin{math}[6,\kappa (2)]=\{ 6, 7, \kappa (2)\}\end{math} is a complement of a maximal sublattice with two join-irreducible elements and the interval \begin{math}[2,\kappa (4)]=\{ 2, \kappa (3), \kappa (4)\}\end{math} is a complement of a maximal sublattice with two meet-irreducible elements.}
    \label{fig:JNExample}
  \end{center}
\end{figure}

\begin{question}\label{ques2} Is it always the case for complements of maximal sublattices of convex geometries that the only join-irreducible element(s) is the minimum and the only meet-irreducible elements are the maxima?
\end{question}

We have confirmed in this paper (Section \ref{CG}) Hypotheses \ref{hyp2}, \ref{hyp3} and \ref{hyp4}, the last of which is formulated at the end of Section \ref{CG}, for CGs with \begin{math}cdim=2\end{math} and we have given a positive answer to Question \ref{ques2} for CGs with \begin{math}cdim=2\end{math}.

\subsection{Plan of the paper}\label{plan}  In Section \ref{preliminaries} we provide the necessary background for (finite) lattices in general, including specifics for the classes we are interested in: SD\begin{math}_\vee\end{math} lattices, SD\begin{math}_\wedge\end{math} lattices, SD lattices  and CGs.
Section \ref{observations} covers useful observations on complements of maximal sublattices in general case.

Section \ref{SDjoin}  discusses complements of maximal sublattices of  SD\begin{math}_\vee\end{math}  lattices. The main results of this section are Theorem \ref{greatestelementinC} and Corollary \ref{coat}, which describe particular situations when the complement of a maximal sublattice in SD\begin{math}_\vee\end{math} is an interval. The dual statements hold for SD\begin{math}_\wedge\end{math} lattices.  Section \ref{SD} covers  SD lattices. Theorem \ref{atomandcoatom} is a direct consequence of Theorem \ref{greatestelementinC} and Corollary \ref{coat} and their dual statements for SD\begin{math}_\wedge\end{math} lattices: if the complement of a maximal sublattice of an SD lattice has a unique maximal or unique minimal element or an atom or coatom, then it is an interval. The main result of the section is Theorem \ref{onecomparable}, which states that if the complement of a maximal sublattice of an SD lattice has one element that is comparable to every other element in the complement, then it is an interval.  Finally, in Section \ref{CG} we describe all complements of maximal sublattices of convex geometries with \begin{math}cdim=2\end{math} (Theorem \ref{BigCGThm}) and  discuss the computational complexity of the algorithm Theorem \ref{BigCGThm} suggests for finding these complements.

\section{Preliminaries}\label{preliminaries}

First we formally establish relevant definitions and known or easily proven facts. For the entirety of this paper we will work with finite lattices which necessarily have a least element and a greatest element, often referred to as 0 and 1, respectively, and in the case of convex geometries, as \begin{math}\emptyset\end{math} and \begin{math}X\end{math}, respectively.

\begin{defn}\label{manydefs} Let \begin{math}\calL\end{math} be a finite lattice.

\begin{enumerate}

\item  An {\em interval} in a lattice \begin{math}\calL\end{math} is a set of the form \begin{math}[a,b]=\{ c: a\leq c\leq b\}\subseteq\calL\end{math}, the {\em down set} of an element \begin{math}b\end{math} is the set \begin{math}\downarrow b=[0,b]\end{math} and the {\em up set} of an element \begin{math}b\end{math} is the set \begin{math}\uparrow b=[b,1]\end{math}.

\item  A subset \begin{math}S\end{math} of lattice \begin{math}\calL\end{math} is called  {\em order convex} if  \begin{math}a<b<c\end{math} and \begin{math}a,c\in  S\end{math} implies \begin{math}b\in S\end{math}; that is, for each \begin{math}a,c\in S\end{math}, \begin{math}[a,c]\subseteq  S\end{math}.\footnote{We note that this use of the phrase {\em order convex} is independent of whether we are working in a Convex Geometry or not.}  If lattice \begin{math}\calL\end{math} is finite, then every order convex sublattice \begin{math}\calS\end{math} is an interval \begin{math}[\bigwedge \calS,\bigvee \calS]\end{math}.

\item An element \begin{math}a\end{math} is a {\em subcover} of \begin{math}b\end{math} in a lattice \begin{math}\calL\end{math} if \begin{math}a<b\end{math} and there is no \begin{math}c \in \calL\end{math} such that \begin{math}a<c<b\end{math}. In this case we also refer to \begin{math}b\end{math} as a {\em cover} of \begin{math}a\end{math} and we write \begin{math}a\prec b\end{math} or \begin{math}b\succ a\end{math}. If \begin{math}a\prec 1\end{math}, then we refer to \begin{math}a\end{math} as a {\em coatom} of the lattice and if \begin{math}0\prec b\end{math} we refer to \begin{math}b\end{math} as an {\em atom} of the lattice.

\item A \begin{math}(0,1)-\end{math}{\em sublattice} of a lattice \begin{math}\calL\end{math} is a sublattice that contains the maximal element and the minimal element of \begin{math}\calL\end{math}.

\item A sublattice \begin{math}\calM\end{math} of a lattice \begin{math}\calL\end{math} is called a {\em maximal sublattice} if \begin{math}\calM\end{math} is a proper sublattice of \begin{math}\calL\end{math} (\begin{math}\calM\neq \calL\end{math}) and whenever \begin{math}\calS\end{math} is a proper sublattice of \begin{math}\calL\end{math} with \begin{math}\calM\subseteq \calS\end{math}, then \begin{math}\calM=\calS\end{math}.

\item\label{sublatticegeneratedby} Given a subset \begin{math}S\end{math} of a lattice \begin{math}\calL\end{math}, the intersection of all the sublattices of \begin{math}\calL\end{math} which contain \begin{math}S\end{math} (the smallest sublattice of \begin{math}\calL\end{math} that contains \begin{math}S\end{math}) is denoted \begin{math}\langle S\rangle\end{math} and is referred to as the {\em sublattice generated by} \begin{math}S\end{math}. If  \begin{math}a\in\calL\end{math}, when there is no ambiguity, we drop the curly brackets when referring to the singleton \begin{math}\{ a\}\end{math} and, for example, for \begin{math}S\subseteq\calL\end{math} and \begin{math}a\in\calL\end{math} we will usually refer to \begin{math}\left< S\cup \{ a\}\right>\end{math} as \begin{math}\left< S\cup a\right>\end{math} to avoid notational clutter. If \begin{math}\calS\end{math} is a maximal sublattice of \begin{math}\calL\end{math} and \begin{math}a\in\calL\setminus\calS\end{math}, then \begin{math}\left< \calS\cup a\right>=\calL\end{math}. This fact is often taken as the definition of a maximal sublattice. For any \begin{math}T\subseteq \calL\end{math} we define
\begin{equation*}
T^{\wedge}=\left\{ \left. \bigwedge_{i=1}^kt_i \ \right| \  k\in\mathbb{Z}^+\text{ and }t_i\in T\right\}, \quad T^{\vee}=\left\{\left. \bigvee_{i=1}^kt_i \ \right| \  k\in\mathbb{Z}^+\text{ and }t_i\in T\right\}.
\end{equation*}

For finite lattices \begin{math}\calL\end{math} we often build \begin{math}\left< S\right>\end{math} inductively in one of two ways:

\begin{enumerate}
    \item  \begin{math}\left< S\right>\end{math} is the limit of the increasing sequence of sets \begin{math}S_n\end{math}, where \begin{math}S_0=S\end{math}, and for \begin{math}n\geq 0\end{math}, \begin{math}S_{n+1}=S_n^{\vee}\cup S_n^{\wedge}\end{math}.

\item    \begin{math}\left< S\right>\end{math} is the limit of the increasing sequence of sets
\begin{equation*}
S_1=S^{\wedge\vee}, \quad S_2=S^{\wedge\vee\wedge\vee}=S_1^{\wedge\vee} , \quad S_3=S^{\wedge\vee\wedge\vee\wedge\vee}=S_2^{\wedge\vee}, \quad \ldots
\end{equation*}

\end{enumerate}

 \item An element \begin{math}j\end{math} of a lattice \begin{math}\calL\end{math} is called {\em join irreducible} if
 \begin{math}j\not = 0\end{math} and \begin{math}x\vee y=j\end{math} implies \begin{math}x=j\end{math} or \begin{math}y=j\end{math}.
 A {\em meet irreducible} element is defined dually.
 An element that is both join and meet irreducible is called {\em doubly irreducible}. The set of join irreducible elements in a lattice \begin{math}\calL\end{math} is denoted \begin{math}Ji(\calL)\end{math}, and the set of meet irreducible elements \begin{math}Mi(\calL)\end{math}.

 Every join irreducible element \begin{math}j\end{math} has a unique subcover, which we will refer to as \begin{math}j_*\end{math} and every meet irreducible element \begin{math}m\end{math} has a unique cover, which we will refer to as \begin{math}m^*\end{math}.

Given any element \begin{math}a\end{math} of  a lattice \begin{math}\calL\end{math}, \begin{math}\bigwedge\{ x\in Mi(\calL): x\geq a\} = a = \bigvee \{y\in Ji(\calL ): y\leq a\}\end{math}, where we define \begin{math}1=\bigwedge \emptyset\end{math} and \begin{math}0=\bigvee \emptyset\end{math}.

 \item A lattice \begin{math}\calL\end{math} is {\em distributive} if for every \begin{math}x,y,z\in\calL\end{math}, \begin{math}x\vee (y\wedge z) = (x\vee y)\wedge (x\vee z)\end{math}. This identity implies the dual one.

\item A lattice \begin{math}\calL\end{math} is {\em join-semidistributive} (SD\begin{math}_{\vee}\end{math})  if   \begin{math}x\vee y = x\vee z\end{math} implies \begin{math}x\vee (y\wedge z) = x\vee y\end{math}.
{\em Meet-semidistributive} lattices are defined dually.

  A lattice is called {\em semidistributive} if it is both meet- and join-semidistributive. We will refer to join-semidistributive lattices as SD\begin{math}_\vee\end{math} lattices, meet-semidistributive lattices as SD\begin{math}_\wedge\end{math} lattices, and semidistributive lattices as SD lattices.

\item\label{def:waybelow} If \begin{math}X,Y\end{math} are two subsets of lattice \begin{math}\calL\end{math}, then \begin{math}X\end{math} is \emph{way-below} \begin{math}Y\end{math}, written \begin{math}X\ll Y\end{math}, if for every \begin{math}x \in X\end{math} there exists \begin{math}y\in Y\end{math} such that \begin{math}x\leq y\end{math}.

\item\label{def:cjmj} An element \begin{math}x \in \calL\end{math} has the \emph{canonical join representation} \begin{math}x=\bigvee X\end{math}, if, for any other join representation \begin{math}x=\bigvee Y\end{math}, one has \begin{math}X\ll Y\end{math}; or equivalently,  \begin{math}Y\ll X\end{math} implies \begin{math}X\subseteq Y\end{math}. In particular, the canonical representation is irreducible: any proper subset \begin{math}Y\subset X\end{math} cannot give \begin{math}\bigvee Y =x\end{math}. Indeed, any subset \begin{math}Y\subseteq X\end{math} satisfies \begin{math}Y\ll X\end{math}, therefore, \begin{math}x=\bigvee Y\end{math} would imply \begin{math}X\subseteq Y\end{math}, or \begin{math}Y=X\end{math}. A canonical meet representation is defined  dually.
All elements in the canonical join representation are join irreducible and those in the canonical meet representation are meet irreducible. We will refer to an element of the canonical join representation of an element \begin{math}x\in\calL\end{math} as a  {\em canonical  joinand} (or {\em cj}) of \begin{math}x\end{math}, and to an element in the canonical meet representation of \begin{math}x\end{math}   as a  {\em canonical meetand} (or {\em cm}) of \begin{math}x\end{math}. (\cite[page 47]{FJN95})

\item Given a finite lattice  \begin{math}\calL\end{math}, \begin{math}j \in Ji(\calL)\end{math}, and \begin{math}m\in Mi(\calL)\end{math}. Let \begin{math}j_*\end{math} be the subcover of \begin{math}j\end{math} and \begin{math}m^*\end{math} be the  cover of \begin{math}m\end{math}. Denote \begin{math}K(j)=\{u \in \calL: u\geq j_*, u\not \geq j\}\end{math} and \begin{math}K^\sigma (m) =\{ v\in \calL: v\leq m^*, v\nleq m\}\end{math}.

Below we list several relevant theorems from \cite{FJN95}. We provide the proof for item (\emph{d}) that can be deduced from other statements.

\begin{thm}\label{Def12Thm} \

\begin{enumerate}

\item\label{cjr} \cite[Theorem 2.24]{FJN95}
 A finite lattice \begin{math}\calL\end{math} satisfies SD\begin{math}_\vee\end{math} if and only if every element of \begin{math}\calL\end{math} has a canonical join representation. Dually, it satisfies SD\begin{math}_\wedge\end{math} if and only if every element has a canonical meet representation.

\item\label{kappa}  \cite[Theorem 2.56]{FJN95} A finite lattice
\begin{math}\calL\end{math} satisfies SD\begin{math}_\wedge\end{math} if and only if for every \begin{math}j\in Ji(\calL)\end{math}, \begin{math}K(j)\end{math} has a greatest element,  \begin{math}\kappa (j)\end{math}, which is in \begin{math}Mi(\calL)\end{math}. Dually, \begin{math}\calL\end{math} satisfies SD\begin{math}_\vee\end{math} if and only if for every \begin{math}m\in Mi(\calL)\end{math}, \begin{math}K^\sigma (m)\end{math} has a least element which we denote \begin{math}\kappa^\sigma (m)\end{math}.

\item\cite[Theorem 2.56]{FJN95}
If \begin{math}\calL\end{math} is SD\begin{math}_\wedge\end{math}, then the mapping \begin{math}\kappa: Ji(\calL) \rightarrow Mi(\calL)\end{math} that is defined due to Theorem \ref{Def12Thm} (b), is onto. Dually, if \begin{math}\calL\end{math} is SD\begin{math}_\vee\end{math} then the mapping \begin{math}\kappa^\sigma: Mi(\calL)\rightarrow Ji(\calL)\end{math} is onto.  If \begin{math}\calL\end{math} is SD, then \begin{math}\kappa\end{math} and \begin{math}\kappa^\sigma\end{math} are inverses of each other and thus both are bijective.

\item\label{SDTypeThms} \cite{A16,FJN95}
Let finite \begin{math}\calL\end{math} satisfy SD\begin{math}_\wedge\end{math} or SD\begin{math}_\vee\end{math}. Then \begin{math}|Ji(\calL)|=|Mi(\calL)|\end{math} if and only if \begin{math}\calL\end{math} satisfies both SD\begin{math}_\wedge\end{math} and SD\begin{math}_\vee\end{math}; that is, if and only if \begin{math}\calL\end{math} is SD.
 In this case,   \begin{math}\kappa(j)=m\end{math} if and only if \begin{math}\kappa^\sigma(m  )=j\end{math}, and \begin{math}j\vee m = m^*\end{math}, \begin{math}j\wedge m=j_*\end{math}.

\end{enumerate}
\end{thm}

\begin{proof} (of \ref{Def12Thm} (d))
If \begin{math}\calL\end{math} is semidistributive, then \begin{math}|Ji(\calL)|=|Mi(\calL)|\end{math} follows from \cite[Corollary 2.55]{FJN95}.

  Suppose we start with an \begin{math}SD_\wedge\end{math} lattice for which \begin{math}|Ji(\calL)|=|Mi(\calL)|\end{math}.
Then for every \begin{math}a \in Ji(\calL)\end{math} there must be \begin{math}\kappa(a) \in Mi(\calL)\end{math}, by \cite[Theorem 2.56]{FJN95}. By part (b) \begin{math}\kappa\end{math} is onto and hence since \begin{math}|Ji(\calL)|=|Mi(\calL)|<\infty\end{math}, \begin{math}\kappa\end{math} is a bijection. Fix \begin{math}c\in Mi(\calL)\end{math}. Let \begin{math}b\end{math} be a minimal element in \begin{math}K^{\sigma}(c)\end{math}. Suppose that \begin{math}b\notin Ji(\calL)\end{math}, i.e. \begin{math}b=x\vee y\end{math}, for \begin{math}x,y\neq b\end{math}. Then, since \begin{math}b\end{math} is minimal in \begin{math}K^{\sigma}(c)\end{math}, neither \begin{math}x\end{math} nor \begin{math}y\end{math} belongs to this set, hence \begin{math}x,y\leq c\end{math} which gives that \begin{math}b\leq c\end{math}, a contradiction. By the same argument, \begin{math}b_*\leq c\end{math} and \begin{math}c\in K(b)\end{math}. Suppose now that \begin{math}c\neq \kappa(b)\end{math}. Then \begin{math}c< \kappa(b)\end{math} which implies \begin{math}c^* \leq \kappa(b)\end{math}, i.e. \begin{math}b\leq \kappa(b)\end{math},
a contradiction. Due to injectivity of the mapping \begin{math}\kappa\end{math}, there is no more than one minimal element in \begin{math}K^\sigma(c)\end{math}, and \begin{math}\kappa^\sigma(c)\end{math} exists. This would imply \begin{math}SD_\vee\end{math} by the dual version of \cite[Theorem 2.56]{FJN95}.
\end{proof}

\item A lattice is {\em  bounded} if it is the image of a bounded homomorphism from a free lattice, see \cite[Chapter II]{FJN95}.

\item\label{CGdef} Although there are many equivalent definitions for {\em convex geometry}, see \cite{AN16}, in this paper we use the following formulation of the definition of a finite  convex geometry.

We start with a finite set \begin{math}X = \{ 1,\ldots , m\}\end{math}, which we will refer to as the {\em base set}, and \begin{math}n\end{math} linear orders on \begin{math}X\end{math} given by \begin{math}n\end{math} permutations \begin{math}\phi_n\end{math} of \begin{math}X\end{math}.  For each linear order \begin{math}<_i\end{math} given by \begin{math}\phi_i(1) <_i\phi_i(2) <_i\cdots <_i \phi_i(m)\end{math}, we define  \begin{math}C_i\end{math} to be the following increasing chain of subsets of \begin{math}X\end{math}:

\begin{equation*}
\emptyset\subseteq \{ \phi_i(1) \} \subseteq \{ \phi_i(1),\phi_i(2)\}  \subseteq \cdots \subseteq \{ \phi_i(1),\ldots , \phi_i(m) \} \ = \ X
\end{equation*}

 It is the case that the set \begin{math}\calG\end{math} of all intersections of sets in the \begin{math}C_i\end{math}, \begin{math}i=1,\ldots , n\end{math}, ordered by set inclusion, is a lattice and we call this lattice the {\em convex geometry} generated by chains \begin{math}C_1\end{math}, \begin{math}C_2\end{math}, \begin{math}\ldots\end{math}, \begin{math}C_n\end{math}. We call the chains \begin{math}C_i\end{math} the {\em generating chains} of \begin{math}\calG\end{math}. In this lattice, \begin{math}a\wedge b=a\cap b\end{math} but we can only generally say \begin{math}a\vee b\supseteq a\cup b\end{math}. Indeed, \begin{math}a\vee b=\bigcap_{a\cup b\subseteq c}c\end{math}. The \begin{math}cdim\end{math} of \begin{math}\calG\end{math} is the minimum number of generating chains needed to generate all the sets in \begin{math}\calG\end{math}, which is also the maximal antichain of meet irreducible elements. In this paper we primarily work with convex geometries with \begin{math}cdim=2\end{math}.  We will sometimes refer to \begin{math}\calG\end{math} as \begin{math}\left< X, \calG \right>\end{math} if \begin{math}X\end{math} is relevant to the discussion.

If \begin{math}\left<X, \calG \right>\end{math} is a convex geometry generated by chains \begin{math}C_1\end{math}, \begin{math}\ldots\end{math}, \begin{math}C_n\end{math}, and \begin{math}C_i\end{math} is given by the order
\begin{math}\phi_i(1) <_i\phi_i(2)<_i\cdots <_i \phi_i(m)\end{math}, then we will refer to the set \begin{math}\{ \phi_i(1), \ldots , \phi_i(k)\}\end{math} as \begin{math}C_i(\phi_i(k))\end{math}.  It is the first element in chain \begin{math}C_i\end{math} that contains  \begin{math}\phi_i(k)\end{math}.   We will refer to the least meet irreducible element of \begin{math}\calG\end{math} that lies on chain \begin{math}C_i\end{math} and contains point \begin{math}x\in X\end{math} as \begin{math}M_i(x)\end{math}. As every element of \begin{math}\calG\end{math} is the meet of elements from the generating chains, for each \begin{math}a\in\calG\end{math} and for each \begin{math}i=1,\ldots ,n\end{math}, there is a least element \begin{math}C_i(a)\end{math} for which \begin{math}a\subseteq C_i(a)\end{math}. Clearly, if \begin{math}a=\{ x\}\end{math} is a singleton, then \begin{math}C_i(a)=C_i(x)\end{math}.

There are a few facts about convex geometries (see \cite{AN16}) that we will be using. Suppose \begin{math}\left< X,\calG\right>\end{math} is a convex geometry with \begin{math}cdim=n\end{math} and generating chains \begin{math}C_1\end{math}, \begin{math}\ldots\end{math}, \begin{math}C_n\end{math}.

\begin{enumerate}
\item  \begin{math}\calG\end{math} is SD\begin{math}_\vee\end{math}.
\item  \begin{math}\calG\end{math} is {\em lower semi-modular}, which means for all \begin{math}x,y\in\calG\end{math}, \begin{math}x\prec x\vee y\Rightarrow x\wedge y\prec y\end{math}; or, equivalently, if \begin{math}x\prec y\end{math}, then for every \begin{math}z\end{math}, either \begin{math}x\wedge z=y\wedge z\end{math} or \begin{math}x\wedge z\prec y\wedge z\end{math}.
\item  If \begin{math}x\prec y\end{math} in \begin{math}\calG\end{math}, then as sets,  \begin{math}|y\setminus x|=1\end{math}.
\item All meet-irreducible elements of \begin{math}\calG\end{math} are on generating chains of \begin{math}\calG\end{math}.
\item For all \begin{math}a\in\calG\end{math}, \begin{math}a=\bigwedge_{i=1}^nC_i(a)=\bigcap_{i=1}^nC_i(a)\end{math}.
\end{enumerate}

\item As discussed in \cite[page 8]{G2},
the {\em glued  sum} of lattices \begin{math}\calL_1\end{math} and \begin{math}\calL_2\end{math} is the lattice \begin{math}\calL\end{math} formed by equating the maximal element of \begin{math}\calL_1\end{math} with the minimal element of \begin{math}\calL_2\end{math}, considering every element
in \begin{math}\calL_1\end{math} to be less than every point in \begin{math}\calL_2\end{math}, and keeping the order on each \begin{math}\calL_i\end{math} the same. This point of ``gluing" will be related to every element of \begin{math}\calL\end{math}. Furthermore, every lattice that has an element, other than its maximal and minimal element, that is comparable
to every other element of the lattice, can be {\em decomposed} into smaller components that are glued-summed together at that point. We  refer  to a lattice as {\em indecomposable} if there is no element other than the least and greatest element that is comparable to every other element of the lattice. Every lattice can be decomposed into a sequence of indecomposable components.

\end{enumerate}
\end{defn}

\section{General Observations}\label{observations}

In this section we note observations about complements of maximal sublattices in general, independent of restrictions we have on the initial lattice \begin{math}\calL\end{math}. Some of these observations give partial answers to \cite[Question 3]{Sch99}:
\begin{question}
Suppose \begin{math}P\end{math} is a finite ordered set occurring as a suborder of a finite lattice. Is there a finite lattice \begin{math}\calL\end{math} and a maximal sublattice \begin{math}\calM\end{math} of \begin{math}\calL\end{math} such that \begin{math}P\cong \calL\setminus \calM\end{math}? How can such ordered sets be characterized?
\end{question}

\begin{obs}\label{obs0} The complement of a maximal sublattice of any lattice will be contained entirely within an indecomposable component of that lattice. In particular, if \begin{math}x\end{math} is comparable to every element of \begin{math}\calL\end{math} and \begin{math}\calS\end{math}  is a  sublattice of \begin{math}\calL\end{math}, then both \begin{math}\calS\cup \uparrow x\end{math} and \begin{math}\calS\cup \downarrow x\end{math} are sublattices of \begin{math}\calL\end{math} and hence if \begin{math}\calS\end{math} is maximal, its complement must be entirely contained in either \begin{math}\uparrow x\end{math} or \begin{math}\downarrow x\end{math}.
\end{obs}

\begin{proof}
  Let \begin{math}\calS\end{math} be a sublattice of \begin{math}\calL\end{math}. We show that \begin{math}\calS\cup\downarrow x\end{math} is a sublattice of \begin{math}\calL\end{math}.  Fix \begin{math}u,v\in\calS\cup\downarrow x\end{math}.  If both \begin{math}u\end{math} and \begin{math}v\end{math} are in \begin{math}\calS\end{math}, then \begin{math}u\vee v\in\calS\end{math} and \begin{math}u\wedge v\in\calS\end{math}. If both \begin{math}u\end{math} and \begin{math}v\end{math} are in \begin{math}\downarrow x\end{math}, then \begin{math}u\vee v\in\downarrow x\end{math} and \begin{math}u\wedge v\in\downarrow x\end{math}. Suppose \begin{math}u\in \calS\setminus \downarrow x\end{math} and \begin{math}v\in \downarrow x\end{math}. Then \begin{math}u\in\uparrow x\end{math} since \begin{math}x\end{math} is comparable to everything and so \begin{math}v\leq x\leq u\end{math}.  That is, \begin{math}u\wedge v=v\in\downarrow x\end{math} and \begin{math}u\vee v=u\in\calS\end{math}.  In all cases, \begin{math}u\vee v\in\calS\cup \uparrow x\end{math} and \begin{math}u\wedge v\in\calS\cup\uparrow x\end{math} and thus \begin{math}\calS\cup \uparrow x\end{math} is a sublattice of \begin{math}\calL\end{math}.  An analogous argument shows that \begin{math}\calS\cup \uparrow x\end{math} is a sublattice of \begin{math}\calL\end{math}.
\end{proof}

\begin{convention} It follows from Observation \ref{obs0} that if the lattice has a unique atom then the complement of any maximal sublattice lies entirely in the up set of that atom and if it has a unique coatom, then the complement lies entirely within the  down set of that coatom.
Thus, going forward we assume all the lattices we are working with have at least two atoms and at least two coatoms.
\end{convention}

\begin{obs}\label{irreducible}
 If \begin{math}\calM\end{math} is a maximal sublattice of lattice \begin{math}\calL\end{math} whose complement \begin{math}\calC=\calL\setminus\calM\end{math} has more than one element, then no element of the complement is doubly irreducible for if \begin{math}a\in\calL\end{math} is doubly irreducible, \begin{math}\{ a\}\end{math} is itself the complement of a sublattice (and hence of a maximal sublattice).
\end{obs}

\begin{obs}\label{maxmeetirreducible}
 If \begin{math}\calS\end{math} is a \begin{math}(0,1)-\end{math}sublattice of lattice \begin{math}\calL\end{math} and \begin{math}a\end{math} is a maximal element of its complement \begin{math}\calC=\calL\setminus\calS\end{math}, then \begin{math}a\end{math} is meet-irreducible. Similarly, if \begin{math}c\end{math} is a minimal element of \begin{math}\calC\end{math}, then \begin{math}c\end{math} is join-irreducible. Thus, if \begin{math}\calC\end{math} has more than one element, then by Observation \ref{irreducible}, no maximal element can be join-irreducible and no minimal element can be meet-irreducible.
\end{obs}

Below \cite[Question 3]{Sch99} there is a comment in the text that Gabriela Bordalo, in a private communication, pointed out that in a complement with at least 2 elements there is no isolated point. There is also a proof (explanation) why it holds. This remark is a special case of our Observation \ref{Cconnected} with one of the components being a singleton. But such an isolated point in a poset is both a maximal and a minimal element in this poset, so it is doubly irreducible and such a situation is already eliminated by Observations \ref{irreducible} and \ref{maxmeetirreducible}.

\begin{obs}\label{Cconnected} If \begin{math}\calC\end{math} is the complement of a maximal sublattice \begin{math}\calM\end{math} of some lattice \begin{math}\calL\end{math}, then \begin{math}\calC\end{math} cannot be partitioned into two nonempty sets \begin{math}C_1\end{math} and \begin{math}C_2\end{math} for which no element of \begin{math}C_1\end{math} is comparable to any element of \begin{math}C_2\end{math}.
  \end{obs}

\begin{proof} Suppose \begin{math}\calM\end{math} is a maximal sublattice of \begin{math}\calL\end{math}, \begin{math}\calC=\calL\setminus\calM\end{math}, and  \begin{math}\calC=C_1\cup C_2\end{math} where no element in \begin{math}C_1\end{math} is comparable to any element of \begin{math}C_2\end{math} and both \begin{math}C_1\end{math} and \begin{math}C_2\end{math} are nonempty. Fix \begin{math}x,y\in \calM\cup C_2\end{math}.  If both \begin{math}x,y\in \calM\end{math}, then necessarily \begin{math}x\wedge y,x\vee y\in\calM\end{math}.  If either \begin{math}x\end{math} or \begin{math}y\end{math} is in \begin{math}C_2\end{math}, then \begin{math}x\wedge y,x\vee y\notin C_1\end{math} since no element in \begin{math}C_2\end{math} is comparable to any element of \begin{math}C_1\end{math}, and thus in either case, \begin{math}x\wedge y,x\vee y\in \calM\cup C_2\end{math}.  That is, \begin{math}\calM\cup C_2\end{math} is a sublattice of \begin{math}\calL\end{math} which contains \begin{math}\calM\end{math}, contradicting the fact that \begin{math}\calM\end{math} is maximal.
\end{proof}

The proof of Observation \ref{obs1} is straightforward.

\begin{obs}\label{obs1} If \begin{math}\calM\end{math} is a maximal sublattice of some lattice \begin{math}\calL\end{math} with at least two atoms then \begin{math}0\in \calM\end{math}. Dually, if \begin{math}\calM\end{math} is a maximal sublattice of some lattice with at least two coatoms then \begin{math}1\in \calM\end{math}.
\end{obs}

It follows, by Observation \ref{obs1}, that in lattices we are interested in our investigations, described above in the  Convention, maximal sublattices are always \begin{math}(0,1)\end{math}-sublattices. It also implies that given any element \begin{math}c\end{math} in the complement \begin{math}\calC\end{math} of a maximal sublattice \begin{math}\calM\end{math}, the sets \begin{math}\{ m: m\in\calM \text{ and }m>c\}\end{math} and  \begin{math}\{ m: m\in\calM \text{ and }m<c\}\end{math} are both nonempty, which allows us the following notation.

\begin{notation}
Let \begin{math}\calM\end{math} be a maximal sublattice of a finite lattice \begin{math}\calL\end{math} with complement \begin{math}\calC=\calL\setminus\calM\end{math}. For each \begin{math}c\in\calC\end{math} we will let
\begin{equation*}
\overline{m}(c) = \bigwedge\{ m: m\in\calM \text{ and }m>c\} \quad \text{and} \quad \underbar{m}(c) = \bigvee\{ m: m\in\calM \text{ and }m<c\}
\end{equation*}
\end{notation}

Recall the definitions of \begin{math}\left< S\right>\end{math}, \begin{math}S^\wedge\end{math} and \begin{math}S^\vee\end{math} from Definition \ref{manydefs} (\ref{sublatticegeneratedby}).
   \begin{obs}\label{1maxORmin}
   Suppose \begin{math}\calM\leq \calL\end{math} is a maximal sublattice with complement \begin{math}\calC=\calL\setminus\calM\end{math}.

   \begin{enumerate}
   \item\label{part36a} If \begin{math}\calC\end{math} has a unique minimal element \begin{math}c_0\end{math} and \begin{math}m_0=\overline{m}(c_0)\end{math}, then \begin{math}\calC\cap [m_0,1]=\emptyset\end{math} and \begin{math}\calC\end{math} is order convex. The dual statement holds for a unique maximal element in \begin{math}\calC\end{math}.
   \item\label{part36b} Let \begin{math}c_1,\dots, c_k\end{math} all be   minimal elements of \begin{math}\calC\end{math}, \begin{math}m_i=\overline{m}(c_i)\end{math}, \begin{math}i\leq k\end{math}, and \begin{math}m_0=\bigwedge m_i\end{math}. Then either \begin{math}\calC \subseteq [m_0,1]\end{math} or \begin{math}\calC\cap [m_0,1]=\emptyset\end{math}. In particular, in the second case, \begin{math}\calC\end{math} is order convex. The dual case holds for maximal elements of \begin{math}\calC\end{math}.
   \item\label{part36c} In the setting of (b), if \begin{math}\calC \subseteq [m_0,1]\end{math}, then \begin{math}\calC\end{math} is a complement of maximal sublattice in \begin{math}[m_0,1]\end{math}.
   \end{enumerate}
   \end{obs}

   \begin{note} Part (\ref{part36c}) tells us that given the setting of (\ref{part36b}), we can restrict ourselves to the case that either \begin{math}m_0=0\end{math} or  \begin{math}\calC\cap [m_0,1]=\emptyset\end{math}.
   \end{note}

   \begin{proof} \

   \begin{enumerate}
\item[(\ref{part36a})] This is a special case of item (b), which we prove more generally below.

\item[(\ref{part36b})]    Suppose \begin{math}c\in\calC\cap [m_0,1]\end{math}. Our goal is to show that given any \begin{math}m\in \calM\end{math}, \begin{math}c\wedge m\end{math} and \begin{math}c\vee m\end{math} are in \begin{math}\calC\cap [m_0,1]\end{math}. Fix \begin{math}m\in\calM\end{math}.  Notice that \begin{math}c\vee m \in [m_0,1]\end{math} since \begin{math}c>m_0\end{math} and if \begin{math}m\in [m_0,1]\end{math}, then \begin{math}m\wedge c\in [m_0,1]\end{math} since both \begin{math}c\end{math} and \begin{math}m\end{math} are bigger than or equal to \begin{math}m_0\end{math}. Suppose now \begin{math}m\ngeq m_0\end{math}. Then \begin{math}m\ngeq m_i\end{math}, and so  \begin{math}m\ngeq c_i\end{math}  for any \begin{math}i=1,\ldots , k\end{math}. But that means \begin{math}m\wedge c\ngeq c_i\end{math} for any \begin{math}i=1,\ldots , k\end{math} and hence \begin{math}c\wedge m\notin\calC\end{math}. This means, \begin{math}\left< \calM, c\right>\subseteq \calM\cup [m_0,1]\end{math}. But since \begin{math}\calM\end{math} is maximal, \begin{math}\calL =\left< \calM, c\right>\end{math} and so \begin{math}\calC\subseteq [m_0,1]\end{math}. Thus, if \begin{math}\calC\cap [m_0,1]\neq\emptyset\end{math} then \begin{math}\calC\subseteq [m_0,1]\end{math}.

Suppose \begin{math}\calC\cap [m_0,1]=\emptyset\end{math}.  Fix \begin{math}c,c'\in\calC\end{math} with \begin{math}c<c'\end{math} and suppose \begin{math}c<b<c'\end{math}.  We know \begin{math}c\geq c_i\end{math} for some \begin{math}i\end{math} and so \begin{math}b>c_i\end{math}, but also, \begin{math}b\ngeq m_0\end{math} since \begin{math}c\ngeq m_0\end{math}.  That is, \begin{math}b\ngeq m_i\end{math} and so \begin{math}b\in\calC\end{math}. That is, \begin{math}\calC\end{math} is order convex.

\item[(\ref{part36c})] Indeed, take \begin{math}\calM^*=\calM\cap [m_0,1]\end{math}. Then     \begin{math}\calC\end{math} is a complement of sublattice \begin{math}\calM^*\end{math} in \begin{math}[m_0,1]\end{math}. We need to show that for any \begin{math}c\in\calC\end{math}, \begin{math}[m_0,1] =\left< \calM^*,c\right>\end{math}. Clearly for any \begin{math}c\in\calC\end{math},  \begin{math}\left< \calM^*,c\right>\subseteq [m_0,1]=\left< \calM , c\right> \cap [m_0,1]\end{math} since \begin{math}\left< \calM, c\right>\end{math} is the entire lattice.

Fix \begin{math}c\in   \calC\end{math}.   Consider \begin{math}\calM_0=\calM\cup c\end{math},
\begin{math}\calM_{n+1}=\calM_n^\vee\cup \calM_n^\wedge\end{math}, \begin{math}\calM_0^*=\calM^*\cup c\end{math}, and \begin{math}\calM_{n+1}^*=\left( \calM_n^*\right)^\vee\cup \left(\calM_n^*\right)^\wedge\end{math} for each \begin{math}n>0\end{math}. Our goal is to show \begin{math}\calM_n\cap [m_0,1]\subseteq\calM_n^*\end{math}. Clearly, \begin{math}\calM_0\cap [m_0,1]\subseteq \calM_0^*\end{math}. Suppose we know that up to a given \begin{math}n\end{math}, \begin{math}\calM_n\cap [m_0,1]\subseteq  \calM_n^* \end{math},
 and choose \begin{math}c'\in \calM_{n+1}\cap [m_0,1]\end{math}. Then there is some \begin{math}c''\in\calM_n\end{math} and some \begin{math}m\in\calM\end{math}, such that either  \begin{math}c'=c''\wedge m\end{math}   or \begin{math}c'=c''\vee m\end{math}. Suppose \begin{math}c'=c''\wedge m\end{math}, then both \begin{math}c''\end{math} and \begin{math}m\end{math} are greater than or equal to \begin{math}m_0\end{math} and so \begin{math}m\in \calM^*\end{math} and \begin{math}c''\in\calM_n^*\end{math}; which means \begin{math}c''\in\calM_{n+1}^*\end{math}.  Suppose instead that \begin{math}c'=c''\vee m\end{math}. Then \begin{math}c'=c''\vee m\vee m_0\end{math} since \begin{math}c'>m_0\end{math}. But \begin{math}m'\in m\vee m_0\in\calM^*\end{math}. Also, \begin{math}c''\in\calC\end{math} and so \begin{math}c''>m_0\end{math}.  That is, \begin{math}c''\in\calM_n\cap [m_0,1]\end{math} and so \begin{math}c'\in \calM_n^*\end{math}.  In either case, \begin{math}c''\in \calM_{n+1}^*\end{math}.
\end{enumerate}
   \end{proof}

 The conclusion of Observation \ref{1maxORmin}(\ref{part36b}) in the case \begin{math}\calC\cap  [m_0,1]=\emptyset\end{math}, contributed to the formulation of Hypothesis \ref{hyp3} about convexity (see Section \ref{S:Intro}).

    \begin{obs}\label{everywhere}
Suppose \begin{math}\calL\end{math} is a  lattice and that \begin{math}\calS\end{math} is a   sublattice of \begin{math}\calL\end{math}.  Suppose also that \begin{math}m'\end{math} is maximal in \begin{math}\calS\setminus 1\end{math} and \begin{math}c\in [m',1]\end{math}. For \begin{math}m\in\calS\end{math}, if \begin{math}m\not\leq m'\end{math} then \begin{math}m\vee c=1\end{math}. The dual statement holds as well.
   \end{obs}

   \begin{proof} Suppose the hypothesis and fix \begin{math}m\in\calS\end{math} and \begin{math}m\not\leq m'\end{math}. Then \begin{math}m\vee m'\in \calS\end{math} and \begin{math}m'\vee m=1\end{math}.  Thus, \begin{math}c\vee m\geq m'\vee m=1\end{math} and so \begin{math}c\vee m=1\end{math}.
  \end{proof}

\begin{obs}\label{maxm}
Suppose \begin{math}\calC\end{math} is the complement of a maximal sublattice \begin{math}\calM\end{math} of some lattice \begin{math}\calL\end{math}, \begin{math}|\calC|\geq 2\end{math} and that \begin{math}\calC\end{math} has a greatest element \begin{math}a\end{math}. Then \begin{math}\underbar{m}(a)\prec a\end{math}. The dual statement holds as well.
\end{obs}

\begin{proof}
If \begin{math}\underbar{m}(a)\end{math} is not a subcover of \begin{math}a\end{math}, then there exists \begin{math}b \in \calC\end{math} such that  \begin{math}\underbar{m}(a) < b\prec a\end{math}. We are going to show that \begin{math}\langle \calM \cup\ b \rangle \subseteq \calM\cup \downarrow b\end{math}, which contradicts the maximality of sublattice \begin{math}\calM\end{math}.

Indeed, \begin{math}\langle \calM \cup\ b \rangle = \bigcup_{n < \omega} \calM_n\end{math}, where \begin{math}\calM_0=\calM\cup\ b\end{math}, \begin{math}\calM_{n+1}=\calM_n^\vee \cup \calM_n^\wedge\end{math}. We prove by induction that \begin{math}\calM_n\subseteq \calM\cup \downarrow b\end{math}. Trivial for \begin{math}\calM_0\end{math}. Now suppose \begin{math}\calM_n\subseteq \calM\cup \downarrow b\end{math}, for some \begin{math}n\end{math}.

If we take \begin{math}x,y \in \calM_n\end{math}, where at least one of the elements, say, \begin{math}x\end{math} is not less than or equal to \begin{math}a\end{math}, then \begin{math}x \in \calM\end{math} and \begin{math}x\vee y \in \calM\end{math}, due to the fact that \begin{math}a\end{math} is the greatest element of \begin{math}\calC\end{math}.

If both \begin{math}x,y \leq a\end{math}, then each is either in \begin{math}\downarrow b\end{math} or in \begin{math}\calM\end{math}.  If either is in \begin{math}\calM\end{math}, then it is less than or equal to \begin{math}\underbar{m}(a)\end{math}, which itself is less than \begin{math}b\end{math}. Either way each is less than or equal to \begin{math}b\end{math}. Therefore, \begin{math}x\vee y \leq b\end{math}. This shows that \begin{math}x\vee y \in \calM\cup \downarrow b\end{math}.

For \begin{math}x\wedge y\end{math},  if either \begin{math}x\end{math} or \begin{math}y\end{math} is in \begin{math}\downarrow b\end{math}, then \begin{math}x\wedge y \in \downarrow b\end{math}. If both are in \begin{math}\calM\end{math} then \begin{math}x\wedge y \in \calM\end{math}.

This shows that \begin{math}\calM_{n+1} \subseteq \calM\cup \downarrow b\end{math}, and we are done.
\end{proof}

 \section{SD\begin{math}_\vee\end{math}  and SD\begin{math}_\wedge\end{math} Lattices}\label{SDjoin}

Here we summarize our results for SD\begin{math}_\vee\end{math} and SD\begin{math}_\wedge\end{math} lattices. We remind you of our hypothesis involving SD\begin{math}_\vee\end{math} and SD\begin{math}_\wedge\end{math} lattices: the complements of maximal sublattices of SD\begin{math}_\vee\end{math} lattices are unions of intervals with the same lower bound and hence have a unique minimal element and dually, the complements of maximal sublattices of SD\begin{math}_\wedge\end{math} lattices are unions of intervals having the same upper bound and hence they have a unique maximal element.  The example shown in Figure \ref{fig:maxnumbermaxelements} from Section \ref{CG} shows that we cannot expect complements of maximal sublattices of SD\begin{math}_\vee\end{math} and SD\begin{math}_\wedge\end{math} lattices to be intervals, in general. This shows that one can find convex geometries with \begin{math}cdim=n\end{math} that have complements of maximal sublattices with \begin{math}n\end{math} maximal elements. Since all convex geometries are SD\begin{math}_\vee\end{math}, one can have arbitrarily many maximal elements of the complement of a maximal sublattice of an SD\begin{math}_\vee\end{math} lattice. Dually, one can have arbitrarily many minimal elements of a complement of a maximal sublattice of an SD\begin{math}_\wedge\end{math} lattice.

Recall Definition \ref{manydefs} (\ref{def:cjmj}) of {\em canonical joinand}  and {\em canonical meetand}. For our analysis of SD\begin{math}_\vee\end{math} and SD\begin{math}_\wedge\end{math} lattices the concepts of {\em strict} canonical joinands and meetands will be useful.

\begin{defn}\label{def:scjscm} Given \begin{math}\calC=\calL\setminus \calS\end{math}, for a sublattice \begin{math}\calS\end{math} of a lattice \begin{math}\calL\end{math}. Then if \begin{math}j\end{math} is a canonical joinand of \begin{math}x\in \calC\end{math}, we call \begin{math}j\end{math} a \emph{strict canonical joinand} (or {\em scj}) with respect to \begin{math}\calC\end{math} if \begin{math}[j,x]\subseteq \calC\end{math}. Similarly, if \begin{math}k\end{math} is  a canonical meetand of \begin{math}x\end{math}, then we call \begin{math}k\end{math} a {\em strict canonical meetand} (or {\em scm}) of \begin{math}x\end{math} with respect to \begin{math}\calC\end{math} if \begin{math}[x,k]\subseteq \calC\end{math}.  (When \begin{math}\calC\end{math} is obvious from context, we will omit the phrase ``with respect to \begin{math}\calC\end{math}.'')
\end{defn}

\begin{lem}\label{onlyCincanonicalrep2}
  Let \begin{math}\calC=\calL\setminus \calS\end{math}, for some sublattice \begin{math}
  \calS\end{math} of an SD\begin{math}_\vee\end{math} lattice \begin{math}\calL\end{math}. Every \begin{math}x \in \calC\end{math} has at least one strict canonical joinand. Dually, if \begin{math}\calL\end{math} is an SD\begin{math}_\wedge\end{math} lattice, then every element in \begin{math}\calC\end{math} has a strict canonical meetand.
  \end{lem}

  \begin{proof}
  Since \begin{math}\calL\end{math} is SD\begin{math}_\vee\end{math}, every element \begin{math}x\end{math} has a canonical join representation. Fix \begin{math}x\in \calC\end{math} with canonical join representation \begin{math}x=\bigvee J\end{math}. If none of \begin{math}j\in J\end{math} is strict, then there exists \begin{math}s_j\in \calS\cap [j,x]\end{math}, for each \begin{math}j \in J\end{math}, therefore, \begin{math}x = \bigvee_{j\in J} s_j \in \calS\end{math}, a contradiction.
  \end{proof}

 Let \begin{math}\calL\end{math} be a lattice. One can show that \begin{math}\calL\end{math} is SD\begin{math}_{\vee}\end{math} if and only if it satisfies the following condition
\begin{displaymath}
  w=\bigvee_{i}a_i=\bigvee_{j}b_j \text{ implies } w=\bigvee_{i,j}(a_i\wedge b_j).
\end{displaymath}
The dual statement holds for any SD\begin{math}_{\wedge}\end{math} lattice.

  \begin{lem}\label{L++}
  Suppose that in an SD\begin{math}_{\vee}\end{math} lattice we have \begin{math}y=\bigvee_{i=1}^n\bigwedge_{j=1}^{k_i}a_i^j\end{math}. If for each \begin{math}(j_1,\ldots ,j_n)\in\mathbb{N}^n\end{math} with \begin{math}1\leq j_i\leq k_i\end{math} we have \begin{math}\bigvee_{i=1}^na_i^{j_i}  = z\end{math}, then \begin{math}z=y\end{math}. The dual statement holds in an SD\begin{math}_{\wedge}\end{math} lattice.
   \end{lem}

   Figure \ref{fig:lem4.3} illustrates the point of this lemma.

\begin{figure}[hbt]
  \begin{center}
    \includegraphics[width=0.5\linewidth]{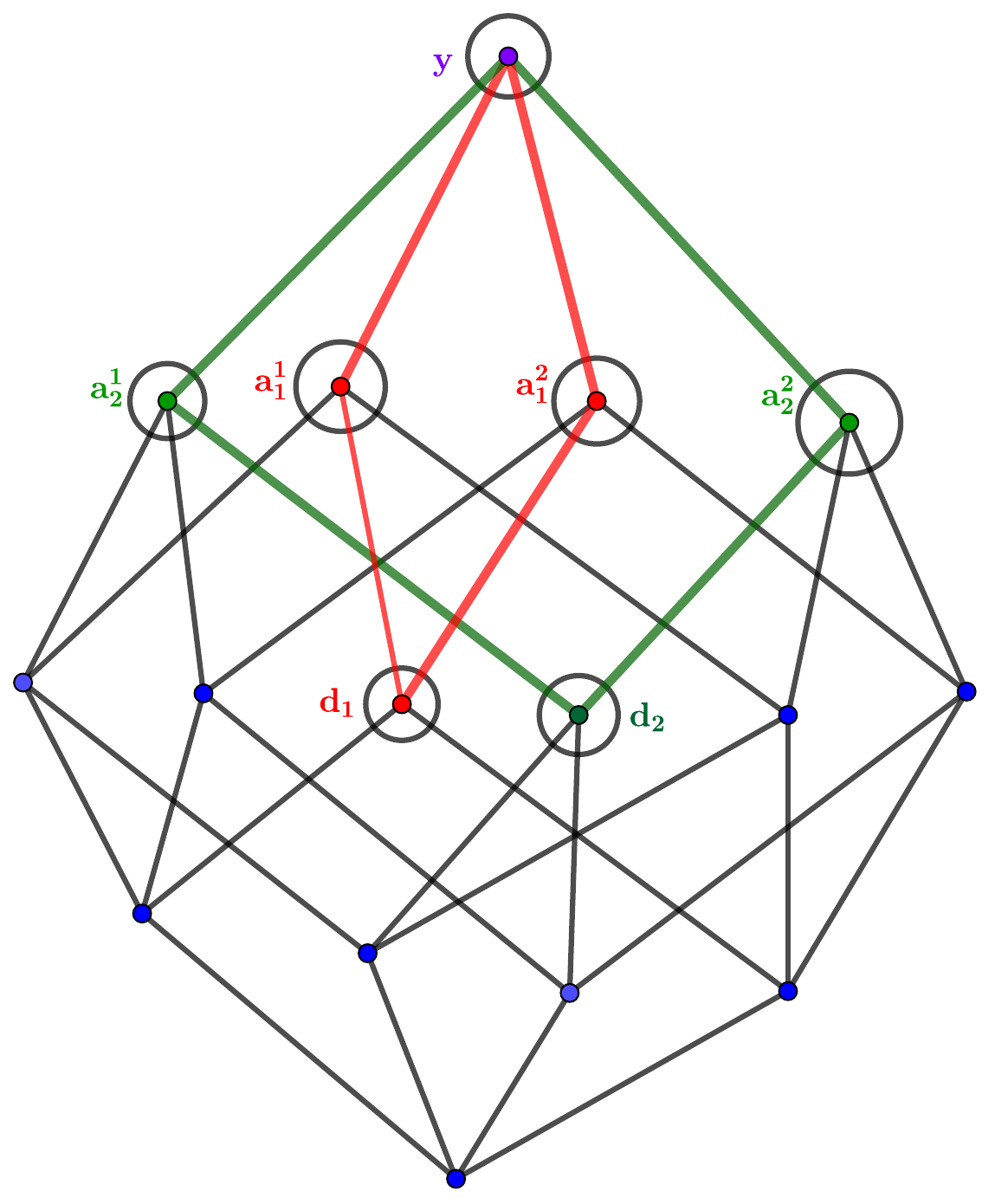}
     \caption{This graphic is to illustrate Lemma \ref{L++}. This lattice happens to be a CG with \begin{math}cdim=4\end{math} and hence is SD\begin{math}_\vee\end{math}. We see that \begin{math}d_1=a_1^1\wedge a_1^2\end{math}, \begin{math}d_2=a_2^1\wedge a_2^2\end{math}, and \begin{math}y=a_1^1\vee a_2^1=a_1^1\vee a_2^2=a_1^2\vee a_2^1=a_1^2\vee a_2^2\end{math}. Thus, Lemma \ref{L++} then guarantees what is shown, that \begin{math}y=d_1\vee d_2\end{math}. }
    \label{fig:lem4.3}
  \end{center}
\end{figure}
   \begin{proof} Let \begin{math}d_i\end{math} denote \begin{math}\bigwedge_{j=1}^{k_i}a_i^j\end{math} for \begin{math}i=1,\ldots ,n\end{math}. Then \begin{math}y=\bigvee_{i=1}^nd_i\end{math}.  Clearly \begin{math}y\leq z\end{math} since \begin{math}d_i\leq a_i^j\end{math} for each \begin{math}j=1,\ldots ,k_i\end{math}.  Let \begin{math}z=\bigvee_{s\in I} w_s\end{math} be the canonical join representation for \begin{math}z\end{math}.  We will show that for each \begin{math}s\in I\end{math}  there is \begin{math}i\end{math} with \begin{math}w_s\leq d_i\end{math}, which implies \begin{math}z\leq y\end{math}.

Fix \begin{math}w_t\end{math}, for some \begin{math}t\in I\end{math}. If for some \begin{math}i=1,\ldots ,n-1\end{math} we have \begin{math}w_t \leq a_i^j\end{math} for all \begin{math}j=1,\ldots ,k_i\end{math}, then \begin{math}w_t\leq \bigwedge_{j=1}^{k_i}a_i^j=d_i\end{math}, and we are done. Therefore, without loss of generality assume that \begin{math}w_t\not \leq a_i^1\end{math} for each \begin{math}i=1,\ldots ,n-1\end{math}.   By hypothesis, \begin{math}\bigvee_{i=1}^{n} a_i^1\vee a_n^j =z\end{math} and so  \begin{math}\bigvee_{s\in I} w_s \ll \bigvee_{i=1}^{n-1} a_i^1\vee a_n^j\end{math} for each \begin{math}j=1,\ldots , k_n\end{math}. Thus, we can conclude that \begin{math}w_t\leq a_n^j\end{math} for each such \begin{math}j\end{math}. That is, \begin{math}w_t\leq d_n\end{math}.  Thus, for each \begin{math}w\in I\end{math}, \begin{math}w_s\leq d_i\end{math} for at least one \begin{math}i=1,\ldots n\end{math} and so \begin{math}z\leq y\end{math}.
\end{proof}

Recall Definition \ref{manydefs}(6) of \begin{math}T^{\wedge}\end{math} and \begin{math}T^{\vee}\end{math}. Given a subset \begin{math}S\end{math} of a lattice \begin{math}\calL\end{math}, the sublattice \begin{math}\left<S\right>\end{math} generated by \begin{math}S\end{math} can be obtained as the limit of the increasing sequence of sets \begin{math}S_1=S^{\wedge\vee}\end{math}, \begin{math}S_2=S^{\wedge\vee\wedge\vee}=S_1^{\wedge\vee}\end{math}, \begin{math}S_3=S^{\wedge\vee\wedge\vee\wedge\vee}=S_2^{\wedge\vee}\end{math}, \begin{math}\ldots\end{math}. We note that any such \begin{math}S_n\end{math} is closed under join.

\begin{thm}\label{gist} Suppose that \begin{math}\mathcal{L}\end{math} is an SD\begin{math}_{\vee}\end{math} lattice, that \begin{math}\mathcal{S}\end{math} is a sublattice of \begin{math}\mathcal{L}\end{math}, that \begin{math}m'\end{math} is maximal in \begin{math}\calS\setminus 1\end{math} and that \begin{math}c\in (m',1)\end{math} is a coatom.   Then for any \begin{math}d\in\calL\end{math},  \begin{math}c \in\left<\calS\cup\ d\right>\end{math} if and only if \begin{math}m'\vee d=c\end{math} and hence \begin{math}d\leq c\end{math}.  Thus, in this case  if also \begin{math}\calS\end{math} is a maximal sublattice, then \begin{math}c\end{math} is the unique maximal element of \begin{math}\calC\end{math}. The dual statement holds for SD\begin{math}_{\wedge}\end{math} lattices.
\end{thm}

\begin{proof} Assume the hypothesis and fix \begin{math}d\in\calL\end{math}. Clearly if \begin{math}m'\vee d=c\end{math} then \begin{math}c\in\left<\calS\cup d\right>\end{math}. Suppose \begin{math}m'\vee d\not=c\end{math}. We may also assume that  \begin{math}d\not= 1\end{math} since \begin{math}\left<\calS\cup 1\right> =\calS\cup 1\end{math}. Then \begin{math}d\not> c\end{math}. Let \begin{math}S_0= \calS\cup d\end{math} and for \begin{math}n\geq 0\end{math} let \begin{math}S_{n+1}=S_n^{\wedge\vee}\end{math}.

We will show \begin{math}c\not\in S_n\end{math} for any \begin{math}n\end{math}.

(basis) Clearly \begin{math}c\notin \calS\cup d\end{math}.

(inductive step) Suppose we know that \begin{math}c\not\in S_n\end{math} for some \begin{math}n\geq 0\end{math} and suppose by way of contradiction that \begin{math}c\in S_{n+1}\end{math}.  Then \begin{math}c=\bigvee_{i=1}^ku_i\end{math} for some \begin{math}u_i=\bigwedge_{j\in J_i}u_i^j\end{math},
where \begin{math}u_i^j\in S_n\end{math} for each \begin{math}j\in J_i\end{math}. If it is the case that for each \begin{math}k-\end{math}tuple \begin{math}(j_1,\ldots , j_k)\in\prod_{i=1}^kJ_i\end{math} we have \begin{math}\bigvee_{i=1}^ku_i^{j_i}=1\end{math}, then by Lemma \ref{L++} we know \begin{math}c=\bigvee_{i=1}^ku_i=1\end{math}, which is a contradiction. So there must be some \begin{math}k-\end{math}tuple \begin{math}(j_1,\ldots , j_k)\in\prod_{i=1}^kJ_i\end{math} such that \begin{math}\bigvee_{i=1}^ku_i^{j_i}<1\end{math}. Since \begin{math}S_n\end{math} is closed under join, by the inductive hypothesis \begin{math}\bigvee_{i=1}^ku_i^{j_i}\not= c\end{math}. However, since \begin{math}u_i\leq u_i^{j_i}\end{math} for each \begin{math}i=1,\ldots ,k\end{math}, we see that \begin{math}c =\bigvee_{i=1}^ku_i\leq \bigvee_{i=1}^k\bigvee_{j\in J_i}u_i^{j_i}\end{math}, which is a contradiction since   \begin{math}c\end{math} is a coatom and  \begin{math}\bigvee_{i=1}^ku_i^{j_i}\end{math} is neither \begin{math}c\end{math} nor   1. Therefore, \begin{math}c\not\in S_{n+1}\end{math}, either.
\end{proof}

\begin{thm}\label{greatestelementinC} Let \begin{math}\calL\end{math} be an SD\begin{math}_\vee\end{math} lattice and \begin{math}\calM\end{math} be a maximal sublattice of \begin{math}\calL\end{math}. If there is a greatest   element in the complement \begin{math}\calC=\calL\setminus\calM\end{math}, then \begin{math}\calC\end{math} is an interval. The dual statement holds for SD\begin{math}_\wedge\end{math} lattices.
\end{thm}

\begin{proof}
Suppose \begin{math}a\end{math} is a greatest element in \begin{math}\calC\end{math}. If there is no element of \begin{math}\calM\end{math} that is less than \begin{math}a\end{math}, then \begin{math}\calC=[0,a]\end{math} and we are done. Suppose that \begin{math}\calM\cap \downarrow a\neq \emptyset\end{math}. Then, by Observation \ref{maxm} we see that \begin{math}\underbar{m}(a)\prec a\end{math} and by Observation \ref{1maxORmin}(\ref{part36a}) we know that \begin{math}\downarrow \underbar{m}(a) \subseteq\calM\end{math}.

Now we want to show that \begin{math}\calC\end{math} is closed under \begin{math}\wedge\end{math}. Suppose by way of contradiction that there exist \begin{math}b,c\in\calC\end{math} such that \begin{math}b\wedge c\in \calM\end{math}.  Then, since \begin{math}b,c\leq a\end{math}, we know \begin{math}b\wedge c \leq a\end{math} and so \begin{math}b\wedge c \leq \underbar{m}(a)\end{math}. But then \begin{math}\underbar{m}(a)\vee (b\land c) =\underbar{m}(a)\neq a\end{math}. However, since \begin{math}\underbar{m}(a)\prec a\end{math}, we also know that \begin{math}a=b\vee \underbar{m}(a)=c\vee \underbar{m}(a)\end{math}.     This contradicts the assumption that \begin{math}\calL\end{math} is join-semidistributive. Therefore, \begin{math}\calC\end{math} must be closed under \begin{math}\wedge\end{math}.  Thus,  \begin{math}d=\bigwedge \calC\in \calC\end{math}. We see that \begin{math}\calC\subseteq [d,a]\end{math}.

Now, suppose that \begin{math}d\leq f\leq a\end{math} for some \begin{math}f\in\calM\end{math}.  Then \begin{math}f\leq \underbar{m}(a)\end{math} and so \begin{math}d\leq \underbar{m}(a)\end{math}, which contradicts the fact that \begin{math}\downarrow \underbar{m}(a)\subseteq \calM\end{math}.  Thus, \begin{math}[d, a]\subseteq \calC\end{math}. That is, \begin{math}\calC=[d,a]\end{math}.
\end{proof}

Corollary \ref{coat} follows immediately from Theorems \ref{greatestelementinC} and \ref{gist}.

\begin{cor}\label{coat}
Let \begin{math}\calL\end{math} be an SD\begin{math}_\vee\end{math} lattice, \begin{math}\calM\leq \calL\end{math} be a maximal sublattice and \begin{math}a\end{math} be a coatom in \begin{math}\calL\end{math}. If \begin{math}a\in\calC=\calL\setminus \calM\end{math}, then \begin{math}\calC\end{math} is an interval.
\end{cor}

\begin{proof} Suppose \begin{math}\calL\end{math} is SD\begin{math}_\vee\end{math}, \begin{math}\calM\leq \calL\end{math} is a maximal sublattice with complement \begin{math}\calC\end{math}, and suppose that \begin{math}a\in\calC\end{math} is a coatom. By Theorem \ref{gist}, \begin{math}a\end{math} is the unique maximal element in   \begin{math}\calC\end{math}. Thus, by Theorem \ref{greatestelementinC}, \begin{math}\calC\end{math} is an interval.

\end{proof}

\section{Semidistributive Lattices}\label{SD}

Our main theorems in this section are Theorem \ref{atomandcoatom}, an immediate consequence of Theorem \ref{greatestelementinC} and Corollary \ref{coat} and their dual statements for SD\begin{math}_\wedge\end{math} lattices,
and Theorem \ref{onecomparable}, which shows that if there is one element of the complement of a maximal sublattice of an SD lattice that is comparable to every other element of the complement, then the complement is an interval.

    \begin{thm}\label{atomandcoatom} Suppose \begin{math}\calL\end{math} is an SD lattice and that \begin{math}\calM\end{math} is a maximal sublattice of \begin{math}\calL\end{math} with complement \begin{math}\calC\end{math}. If \begin{math}\calC\end{math} contains either a greatest or least element or if it contains either an atom or a coatom of \begin{math}\calL\end{math}, then \begin{math}\calC\end{math} is an interval.
  \end{thm}

 \begin{lem}\label{chain4}
 Assume that \begin{math}\calL\end{math} is an SD lattice, where \begin{math}w\leq x \leq y \leq z\end{math}, and that \begin{math}x\end{math} is a cj for \begin{math}z\end{math} and \begin{math}y\end{math} is a cm for \begin{math}w\end{math}. Then
\begin{math}\calL\setminus [x,y]\end{math} is a sublattice of \begin{math}\calL\end{math}.
 \end{lem}

 \begin{proof}
 If \begin{math}x=z\end{math}, then \begin{math}x=y\end{math} and thus \begin{math}x\end{math} is a doubly irreducible element, therefore, \begin{math}\calL\setminus x\end{math} is a sublattice. Similarly, we get a conclusion when \begin{math}y=w\end{math}. Therefore, assume that \begin{math}x<z\end{math} and \begin{math}w < y\end{math}.

 Since \begin{math}x\end{math} is a cj of \begin{math}z\end{math}, \begin{math}z=x\vee b\end{math}, where \begin{math}b\end{math} is the join of all canonical joinands of \begin{math}z\end{math} different from \begin{math}x\end{math}. Since \begin{math}y\end{math} is a cm of \begin{math}w\end{math}, \begin{math}w=y\wedge b'\end{math}, where \begin{math}b'\end{math} is the meet of all canonical meetands of \begin{math}w\end{math} different from \begin{math}y\end{math}.

 We will show that if \begin{math}m_1,m_2 \in  \calL_1=\calL\setminus [x,y]\end{math}, then \begin{math}m_1\vee m_2, m_1\wedge m_2 \in \calL_1\end{math}.

 Assume that \begin{math}m_1\vee m_2\end{math} is in \begin{math}[x,y]\end{math}. Then \begin{math}y\geq m_1, m_2\end{math}.  Also, \begin{math}z=m_1\vee m_2\vee b\end{math}. Therefore, \begin{math}x\leq m_1\end{math} or \begin{math}x\leq m_2\end{math}, which would imply \begin{math}m_1\in [x,y]\end{math} or \begin{math}m_2\in [x,y]\end{math}, a contradiction.

 Symmetrically, assume that \begin{math}m_1\wedge m_2\end{math} is in \begin{math}[x,y]\end{math}. Then \begin{math}x\leq m_1, m_2\end{math}.  Also, \begin{math}w=m_1\wedge m_2\wedge b'\end{math}. Therefore, \begin{math}y\geq m_1\end{math} or \begin{math}y\geq m_2\end{math}, which would imply \begin{math}m_1\in [x,y]\end{math} or \begin{math}m_2\in [x,y]\end{math}, a contradiction.

 Therefore, \begin{math}\calL\setminus [x,y]\end{math} must be a sublattice.
\end{proof}

  Corollary \ref{sclinear} is a direct consequence of Lemma \ref{chain4} (and its dual) and says that whenever the complement of a maximal sublattice, \begin{math}\calC\end{math} has a linear chain whose (strict) joinand/meetand relationship is ``spiraling out'', as illustrated in Figure \ref{fig:5.3}, then \begin{math}\calC\end{math} must be an interval.

\begin{figure}[hbt]
  \begin{center}
    \includegraphics[width=0.5\linewidth]{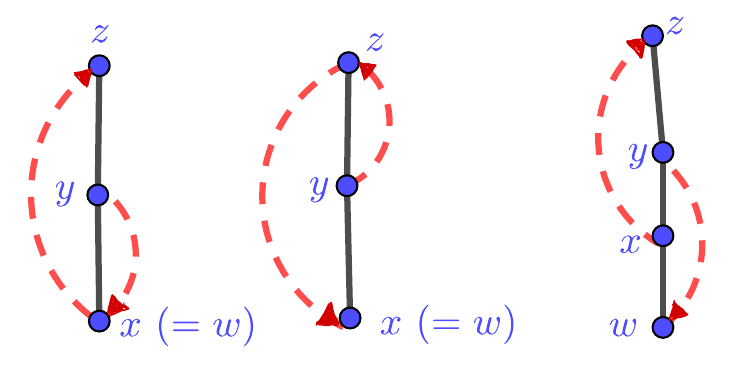}
     \caption{This figure illustrates Corollary \ref{sclinear}.  In the figure, there is a dotted arrow from point \begin{math}a\end{math} to point \begin{math}b\end{math} if \begin{math}a\end{math} is a cj or a cm of \begin{math}b\end{math}, and the corollary says that if the complement of a maximal sublattice has such a ``spiraling out'' chain, then it is an interval. }
    \label{fig:5.3}
  \end{center}
\end{figure}

 \begin{cor}\label{sclinear} Given that \begin{math}\calL\end{math} is an SD lattice, where \begin{math}w\leq x \leq y \leq z\end{math} and that \begin{math}x\end{math} is a  cj for \begin{math}z\end{math} and \begin{math}y\end{math} is a cm of \begin{math}w\end{math}.  Suppose also that \begin{math}\calM\end{math} is a maximal sublattice of \begin{math}\calL\end{math} with complement \begin{math}\calC\end{math} and \begin{math}x\end{math} is a  scj for \begin{math}z\end{math}. Then \begin{math}y=z\end{math} and \begin{math}\calC =[x,y]\end{math}. Dually, if \begin{math}y\end{math} is a scm of \begin{math}w\end{math}, then \begin{math}w=x\end{math} and \begin{math}\calC=[x,y]\end{math}.
 \end{cor}

 \begin{proof}
 It follows from Lemma \ref{chain4} that \begin{math}\calL\setminus [x,y]\end{math} is a sublattice. By assumption of Corollary \begin{math}[x,z]\subseteq \calC\end{math}, therefore, \begin{math}y=z\end{math} and \begin{math}\calC =[x,y]\end{math}.
 \end{proof}

\begin{lem}\label{OverTbridge} Let \begin{math}\calC\end{math} be the complement of a sublattice
\begin{math}\calS\end{math} of an SD lattice \begin{math}\calL\end{math}. Suppose the following is true: (a) \begin{math}u_1\end{math}, \begin{math}u_2\end{math}, \begin{math}t\end{math} and \begin{math}x\end{math} are elements of \begin{math}\calC\end{math}, (b)  \begin{math}t\end{math} is comparable to every element of \begin{math}\calC \cap [x,u_2]\end{math}, (c) \begin{math}x < t\leq u_i\end{math}, \begin{math}u_2\not \leq u_1\end{math}, and (d) \begin{math}u_1\end{math} is a scm
for \begin{math}x\end{math}. Then \begin{math}[x,u_2]\cap \calS \not =\emptyset\end{math}. In particular, \begin{math}x\end{math} cannot be a scj for \begin{math}u_2\end{math} and \begin{math}u_2\end{math} cannot be a scm for \begin{math}x\end{math}.

Dual statement also holds.
\end{lem}

Figure \ref{fig:lem5.4} illustrates the situation Lemma \ref{OverTbridge} does not allow in \begin{math}\calC\end{math}.

\begin{figure}[hbt]
  \begin{center}
    \includegraphics[width=0.4\linewidth]{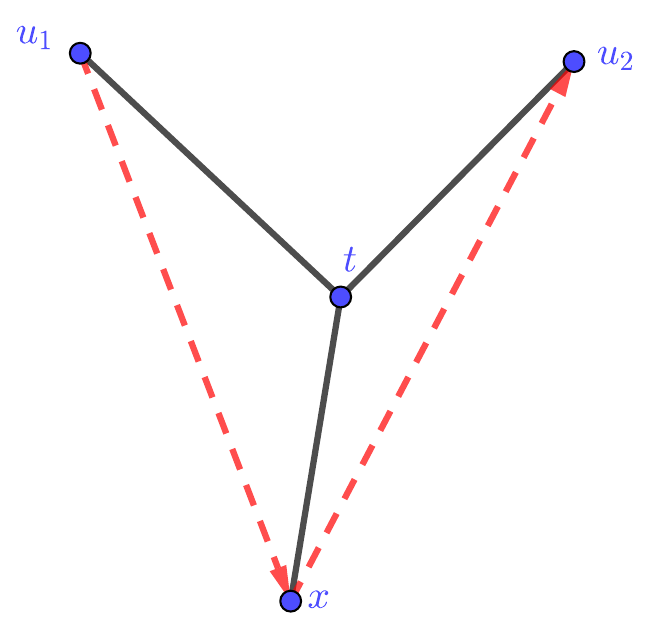}
     \caption{This figure illustrates the situation that Lemma \ref{OverTbridge} does not allow. The dotted arrow from \begin{math}u_1\end{math} to \begin{math}x\end{math} indicates \begin{math}u_1\end{math} is a scm of \begin{math}x\end{math} and the arrow from \begin{math}x\end{math} to \begin{math}u_2\end{math} indicates \begin{math}x\end{math} is a scj of \begin{math}u_2\end{math}. }
    \label{fig:lem5.4}
  \end{center}
\end{figure}

\begin{proof}
Suppose \begin{math}x=u_1\wedge b\end{math}, where \begin{math}b\end{math} is the meet of all of the other meetands of \begin{math}x\end{math}. Take \begin{math}s =u_2\wedge b\end{math}. Then \begin{math}x\leq s \leq u_2\end{math}. Note that since \begin{math}u_1\end{math} is a scm
of \begin{math}x\end{math} and \begin{math}u_2\not \leq u_1\end{math}, we get
\begin{math}x<s\end{math}.

If \begin{math}s\in \calC\end{math}, then either \begin{math}s\geq t\end{math} or \begin{math}s\leq t\end{math}.    If \begin{math}s\geq t\end{math}, then \begin{math}b\geq t\end{math} and so \begin{math}u_1\wedge b\geq t>x\end{math}, a contradiction.  If \begin{math}s\leq t\end{math}, then \begin{math}u_1\geq s\end{math} and so \begin{math}x=u_1\wedge b\geq s\end{math}, also a contradiction.
Therefore, \begin{math}s\in \calS\end{math}.
\end{proof}

\begin{thm} \label{onecomparable} Let \begin{math}\calL\end{math} be an SD lattice and \begin{math}\calM\end{math} be a maximal sublattice of \begin{math}\calL\end{math}. If there is an element \begin{math}a\in\calC=\calL\setminus\calM\end{math} that is comparable to every element of \begin{math}\calC\end{math}, then \begin{math}\calC\end{math} is an interval.
\end{thm}

\begin{proof}  Suppose by way of contradiction that \begin{math}\calC\end{math} is not an interval. By Lemma \ref{onlyCincanonicalrep2} there is a scj \begin{math}c_1\end{math} of \begin{math}a\end{math}. Clearly \begin{math}c_1\leq a\end{math}. Again by Lemma \ref{onlyCincanonicalrep2} there is a scm \begin{math}b_1\end{math} of \begin{math}c_1\end{math}. Since \begin{math}a\end{math} is comparable to every element of \begin{math}\calC\end{math}, either \begin{math}b_1\leq a\end{math} or \begin{math} b_1 > a\end{math}.  If \begin{math}b_1\leq a\end{math}, then we have \begin{math}c_1\leq b_1\leq a\end{math} and by Corollary \ref{sclinear} we know \begin{math}\calC =[c_1,b_1]\end{math}. So we can suppose \begin{math}a<b_1\end{math}. Again by Lemma \ref{onlyCincanonicalrep2} we know there is a scj \begin{math}c_2\end{math} of \begin{math}b_1\end{math}.  Again, since \begin{math}a\end{math} is comparable to everything in \begin{math}\calC\end{math}, either \begin{math}c_2\geq a\end{math} or \begin{math}c_2<a\end{math}. If \begin{math}c_2\geq a\end{math} then we know \begin{math}c_1\leq c_2\leq b_1\end{math} and so by Corollary \ref{sclinear} we know \begin{math}\calC = [c_2,b_1]\end{math}. So we can assume that \begin{math}c_2<a\end{math} and by the same reasoning, \begin{math}c_1\nleq c_2\end{math}.  Since \begin{math}[c_1,b_1]\subseteq \calC\end{math} and \begin{math}[c_2,b_1]\subseteq \calC\end{math}, by Lemma \ref{OverTbridge}, it also cannot be the case that \begin{math}c_1\end{math} and \begin{math}c_2\end{math} are incomparable. Therefore, it must be the case that \begin{math}c_2< c_1\end{math}. This is illustrated with the first graphic in Figure \ref{fig:thm5.5proof}.

\begin{figure}[hbt]
  \begin{center}
    \includegraphics[width=0.35\linewidth]{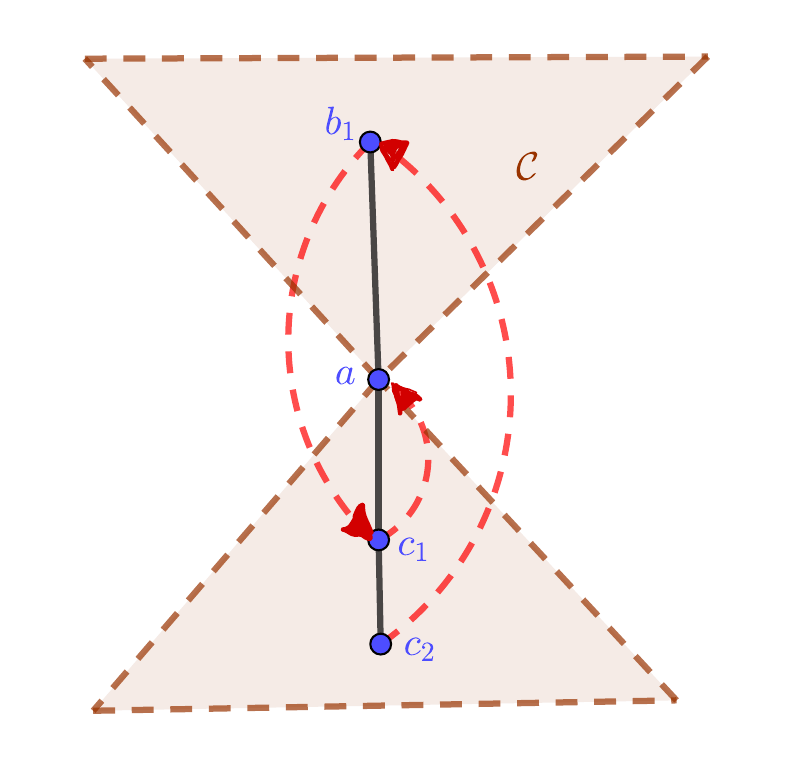} \quad \includegraphics[width=0.4\linewidth]{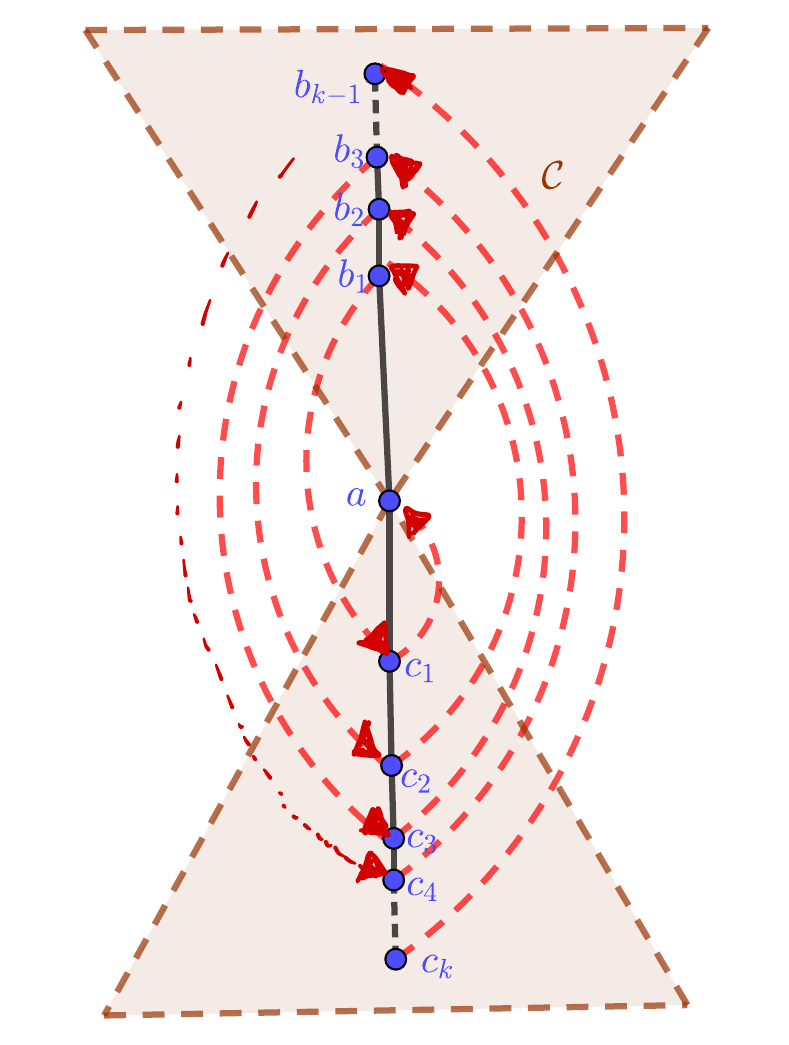}
     \caption{These images illustrate the proof of Theorem \ref{onecomparable}. As before, there is a dotted arrow from point \begin{math}a\end{math} to point \begin{math}b\end{math} if either \begin{math}a\end{math} is a cj or a cm of \begin{math}b\end{math}.}
    \label{fig:thm5.5proof}
  \end{center}
\end{figure}

We proceed inductively.  Suppose we have a sequence \begin{math}c_k< c_{k-1} < \cdots < c_2<c_1\leq a < b_1<b_2<\cdots < b_{k-1}\end{math} and that \begin{math}c_1\end{math} is a scj of \begin{math}a\end{math}, \begin{math}b_j\end{math} is a scm of \begin{math}c_j\end{math} for each \begin{math}j = 1, \ldots , k-1\end{math}, and \begin{math}c_{\ell }\end{math} is a scj of \begin{math}b_{\ell -1}\end{math} for each \begin{math}\ell =2,\ldots k\end{math}. This is illustrated with the  second  graphic in Figure \ref{fig:thm5.5proof}.

We know by Lemma \ref{onlyCincanonicalrep2} that there is a scm \begin{math}b_k\end{math} of \begin{math}c_k\end{math}. If \begin{math}b_k\leq b_{k-1}\end{math}, then \begin{math}c_k\leq b_k\leq b_{k-1}\end{math} and so by Corollary \ref{sclinear} we  have   \begin{math}\calC = [c_k,b_k]\end{math} and so we can assume that \begin{math}b_k\nleq b_{k-1}\end{math}. But then that means also that \begin{math}b_k\nleq a\end{math} and so \begin{math}b_k>a\end{math}.
Since \begin{math}[c_k,b_{k-1}]\subseteq\calC\end{math} and \begin{math}[c_k,b_k]\subseteq\calC\end{math}, by Lemma \ref{OverTbridge} \begin{math}b_k\end{math} and \begin{math}b_{k-1}\end{math} must be comparable. Therefore, \begin{math}b_{k-1}<b_k\end{math}. By the dual argument, there is a scj \begin{math}c_{k+1}\end{math} of \begin{math}b_k\end{math} and \begin{math}c_{k+1}<c_k\end{math}.  Thus, if \begin{math}\calC\end{math} is not an interval, we build an infinite chain in \begin{math}\calC\end{math}, contradicting the fact it is finite. Thus, \begin{math}\calC\end{math} must be an interval.
\end{proof}

\begin{cor}\label{onelgstorsmlst}
Let \begin{math}\calL\end{math} be an SD lattice, \begin{math}\calM\end{math} be a maximal sublattice of \begin{math}\calL\end{math}. If \begin{math}\calC=\calL\setminus\calM\end{math} has a greatest element or a least element, then \begin{math}\calC\end{math} is an interval.
\end{cor}

\section{Convex Geometries}\label{CG}

In this section we prove Theorem \ref{BigCGThm}, which describes precisely what the complements of maximal sublattices of CGs with \begin{math}cdim=2\end{math} can be. As we will see, these complements can have at most 2 maximal elements. As all convex geometries are SD\begin{math}_\vee\end{math} lattices, everything from Section \ref{SDjoin} holds for convex geometries.

Recall that a CG with \begin{math}cdim=n\end{math} has a maximal antichain of  meet irreducible elements of size \begin{math}n\end{math}, and maximal elements in the complement of a maximal sublattice are meet irreducibles. Thus, the complement of a maximal sublattice of a CG  with \begin{math}cdim=n\end{math} can have at most \begin{math}n\end{math} maximal elements itself. Figure \ref{fig:maxnumbermaxelements} shows that we can find CGs of \begin{math}cdim=n\end{math} that have complements of maximal sublattices with exactly \begin{math}n\end{math} maximal elements.
 One can check that \begin{math}[2,12]\cup[2,24]\cup\cdots\cup [2,2(n+2)]\end{math} is the complement of a maximal sublattice with \begin{math}n\end{math} maximal elements.

Let \begin{math}\calG =\left<  X, \calG \right>\end{math} be a convex geometry. If \begin{math}x\in X\end{math}, then we let \begin{math}(x)\end{math} denote the smallest convex set of \begin{math}X\end{math} containing \begin{math}x\end{math}.   It will always be join irreducible in a standard closure system and thus in convex geometries and every join irreducible element of a convex geometry  \begin{math}\calG =\left<  X, \calG \right>\end{math} is \begin{math}(x)\end{math} for some \begin{math}x\in X\end{math}.  Furthermore, in convex geometries if \begin{math}x\neq y\end{math} then \begin{math}(x)\neq (y)\end{math}.

For this section we refer you to the definition of convex geometry and related concepts from Definition \ref{manydefs}(\ref{CGdef}) in Section \ref{preliminaries}.

 \subsection{Convex Geometries with \begin{math}cdim=2\end{math}}

The goal of this subsection is to prove Theorem \ref{BigCGThm}, which describes all the complements of convex geometries with \begin{math}cdim=2\end{math}.

For the entirety of this subsection we will assume \begin{math}\{i,i'\}=\{1,2\}\end{math}. Recall that we refer to the unique cover of a meet irreducible element \begin{math}m\end{math} as \begin{math}m^*\end{math}, we refer to the unique subcover of a join irreducible element \begin{math}j\end{math} as \begin{math}j_*\end{math}.

We build to a proof of the following theorem:

\begin{thm}\label{BigCGThm}
    The complements of maximal sublattices of a convex geometry \begin{math}\calG =\left<  X, \calG \right>\end{math} with \begin{math}cdim=2\end{math} generated by chains \begin{math}C_1\end{math} and \begin{math}C_2\end{math} are precisely sets of one of the three forms:
    \begin{enumerate}

        \item\label{cg1}
       Intervals \begin{math}[(j),C_i(j)]\end{math}, where there is \begin{math}x\in X\end{math} such that \begin{math}C_i(j)\prec C_i(x)\end{math} and \begin{math}C_{i'}(j)>C_{i'}(x)\end{math}. In this case \begin{math}(j)\end{math} is not on \begin{math}C_{i'}\end{math}.

         \item\label{cg2}         Intervals \begin{math}[(j),C_i(j)]\end{math}, where \begin{math}(j)\end{math} is on \begin{math}C_{i'}\end{math} and there is an \begin{math}x\in X\end{math} such that both \begin{math}C_1(j)\prec C_1(x)\end{math} and \begin{math}C_2(j)\prec C_2(x)\end{math}.

        \item\label{cg3}
       Union of intervals \begin{math}[(j),C_1(j)]\cup [(j),C_2(j)]\end{math}, where \begin{math}(j)\end{math} is not on  \begin{math}C_1\end{math} nor \begin{math}C_2\end{math} and there is an \begin{math}x\in X\end{math} such that both \begin{math}C_1(j)\prec C_1(x)\end{math} and \begin{math}C_2(j)\prec C_2(x)\end{math}.

    \end{enumerate}
\end{thm}

    \begin{note} Since the maximal elements of complements of sublattices are necessarily meet irreducible, in cases (\ref{cg1}) and (\ref{cg2}) above, \begin{math}C_i(j)=M_i(j)\end{math} for the appropriate \begin{math}i\end{math} and in case (\ref{cg3}) we have \begin{math}C_i(j)=M_i(j)\end{math} for each \begin{math}i\end{math}.
    \end{note}

If the complement of the maximal sublattice consists of a single doubly irreducible element of the lattice, then it is a special case of item (\ref{cg1}) unless all of \begin{math}(j), (j)_*,(j)^*\in C_1\cap C_2\end{math} in which case it is a special case of item (\ref{cg2}).

\begin{exam}\label{examp:mainCGthmexamp} Figure \ref{fig:CompsOfCG2} illustrates the three types of complements of maximal sublattices of a CG with \begin{math}cdim=2\end{math}. It is generated by chains \begin{math}C_1\end{math} and \begin{math}C_2\end{math}, determined by the ordering \begin{math}<_1\end{math} and \begin{math}<_2\end{math} on the set \begin{math}X=\{ 1,\ldots, 10\}\end{math}, respectively. The singleton \begin{math}\{ (2)\}\end{math} illustrates the complement of a maximal sublattice of type (\ref{cg1}) that consists of a single doubly-irreducible element of the lattice. In this case \begin{math}(2)=C_1(2)\end{math},  \begin{math}C_1(2)\prec C_1(3)\end{math}, and \begin{math}3\in C_2(2)\end{math} (which you can tell because \begin{math}C_2(3)<C_2(2)\end{math}). The interval \begin{math} [(5),C_1(5)]\end{math}  illustrates another complement of a maximal sublattice of type (\ref{cg1}). Notice that   \begin{math}C_1(5)\prec C_1(6)\end{math} while \begin{math}6\in C_2(5)\end{math}. The interval \begin{math}[(6),C_1(6)]\end{math} illustrates a complement of type (\ref{cg2}). We see that  \begin{math}C_i(6)\prec C_i(7)\end{math} for each \begin{math}i\in \{ 1,2\}\end{math} and \begin{math}(6)=C_2(6)\end{math}.  Finally, the union of intervals \begin{math}[(8),C_1(8)]\cup [(8),C_2(8)]\end{math} illustrates   a complement of a maximal sublattice of type (\ref{cg3}). Note \begin{math}(8)\end{math} is in neither chain and \begin{math}C_i(8)\prec C_i(9)\end{math} for each \begin{math}i\end{math}.

\begin{figure}[hbt]
  \begin{center}
    \includegraphics[width=0.55\linewidth]{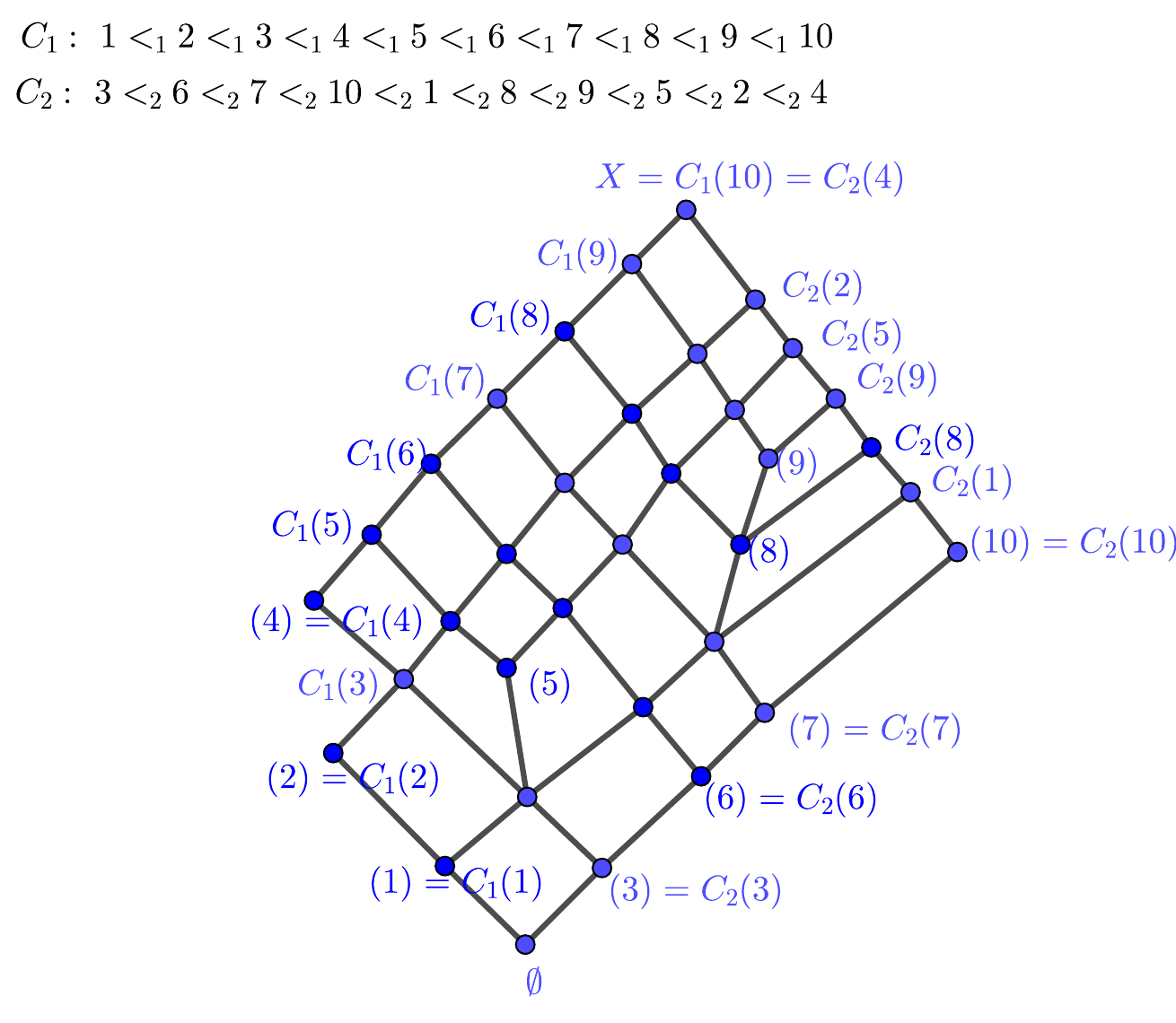}
     \caption{The three types of complements  of maximal sublattices of CGs with \begin{math}cdim=2\end{math}, explained in Example \ref{examp:mainCGthmexamp}.}
    \label{fig:CompsOfCG2}
  \end{center}
\end{figure}

\end{exam}

 If \begin{math}C_1\end{math} and \begin{math}C_2\end{math} are chains generating a convex geometry with \begin{math}cdim=2\end{math} and \begin{math}a\in C_1\cap C_2\end{math}, then \begin{math}a\end{math} is comparable to every element of the convex geometry. Thus, Observation \ref{obs0}  from Section \ref{observations} assures us that we may restrict our analysis to convex geometries of \begin{math}cdim=2\end{math} for which the intersection of the two generating chains contains only the minimal element \begin{math}\emptyset\end{math} and maximal element \begin{math}X\end{math}. We will say chains \begin{math}C_1\end{math} and \begin{math}C_2\end{math} generating a convex geometry \begin{math}\calG =\left<  X, \calG \right>\end{math} of \begin{math}cdim=2\end{math} have {\em trivial intersection} if \begin{math}C_1\cap C_2=\{ \emptyset , X\}\end{math}.

 \begin{rem} Note that generating chains have trivial intersection if and only if \begin{math}X\end{math} cannot be written as the disjoint union of non-empty sets \begin{math}X_1\end{math} and \begin{math}X_2\end{math} such that the orderings that determine both \begin{math}C_1\end{math} and \begin{math}C_2\end{math} have every element of \begin{math}X_1\end{math} as less than every element of \begin{math}X_2\end{math}.
 \end{rem}

 Lemma \ref{JI}
     below summarizes some simple facts about the structure of the generating chains of convex geometries with \begin{math}cdim=2\end{math}.

\begin{lem}\label{JI} Suppose \begin{math}\calG =\left<  X, \calG \right>\end{math} is a convex geometry of \begin{math}cdim=2\end{math} generated by chains \begin{math}C_1\end{math} and \begin{math}C_2\end{math}.  Then each of the following is true.
\begin{enumerate}
\item\label{ji1} For every \begin{math}x\in X\end{math}, \begin{math}(x)=C_1(x)\wedge C_2(x) =M_1(x)\wedge M_2(x)=c_1\wedge c_2\end{math} for each \begin{math}c_i\in [(x), M_i(x)]\end{math}.

\item\label{jiintersection} For each \begin{math}x\in X\end{math}, \begin{math}[(x),M_1(x)]\cap [(x), M_2(x)]=\{ (x)\}\end{math} and so \begin{math}[(x),C_1(x)]\cap [(x), C_2(x)]=\{ (x)\}\end{math}, as well.

\item\label{C=M} Let \begin{math}C_1(x)\prec C_1(x_1)\end{math} and \begin{math}C_2(x)\prec C_2(x_2)\end{math}. If \begin{math}x_1\in C_2(x)\end{math} or \begin{math}x_1=x_2\end{math}, then \begin{math}C_1(x)=M_1(x)\end{math}.

\item\label{ji5}   For any \begin{math}a\in\calG\end{math} which is not the minimal element of \begin{math}\calG\end{math}, \begin{math}C_i(a) = C_i(j)\end{math} for some \begin{math}j\in X\end{math} with \begin{math}(j)\leq a\end{math}.

\item\label{ji10} For any \begin{math}x\in X\end{math}, \begin{math}[\emptyset , C_i(x)] = [  (x),C_i(x)]\dot{\cup} [\emptyset , C_i(p)]\end{math}, where \begin{math}C_i(p)\prec C_i(x)\end{math}.

\item\label{jiCi=X} Suppose \begin{math}C_i(x)=X\end{math} for some \begin{math}x\in X\end{math}.  Then \begin{math}(x)=C_{i'}(x)
\in C_{i'}\end{math} and if \begin{math}(x)\neq X\end{math}, then \begin{math}(x)=M_{i'}(x)\end{math}, as well.

\end{enumerate}
\end{lem}
\begin{proof} Fix a convex geometry \begin{math}\calG \end{math} with \begin{math}cdim=2\end{math}, generating chains \begin{math}C_1\end{math} and \begin{math}C_2\end{math}, and base set \begin{math}X\end{math}.

\begin{enumerate}
    \item[\ref{ji1}.]  Fix \begin{math}x\in X\end{math}. We know that \begin{math}(x)=m_1\wedge m_2\end{math} for some meet irreducible \begin{math}m_i\in C_i\end{math}.  But then, \begin{math}x\in m_i\end{math} and thus, \begin{math}m_i\geq M_i(x)\geq C_i(x)\end{math} for each \begin{math}i\end{math}. Thus, \begin{math}(x)=m_1\wedge m_2\geq M_1(x)\wedge M_2(x)\geq C_1(x)\wedge C_2(x)\end{math}.  But also, \begin{math}x\in C_i(x)\end{math} and hence \begin{math}x\in C_1(x)\wedge C_2(x)\end{math}.   That is, \begin{math}(x)\leq C_1(x)\wedge C_2(x)\end{math} as well.  Hence \begin{math}(x)=C_1(x)\wedge C_2(x)=M_1(x)\wedge M_2(x)\end{math}. Further, if \begin{math}c_i\in [(x), M_i(x)]\end{math}, then \begin{math}(x)\leq c_1\wedge c_2\leq M_1(x)\wedge M_2(x)= (x) \end{math}.

\item[\ref{jiintersection}] Let \begin{math}a\in [(x),M_1(x)]\cap [ (x), M_2(x)]\end{math}. Then \begin{math}a\in [(x), M_2(x)]\end{math} and \begin{math}a\leq M_1(x)\end{math}. Thus,  by part (\ref{ji1}),  \begin{math}(x)=a\wedge M_1(x) =a\end{math}.

    \item[\ref{C=M}]  Denote \begin{math}a=C_1(x)\end{math}. It is enough to show that \begin{math}C_2(a) \geq C_1(x_1)=a\cup \{x_1\}\end{math}, i.e., \begin{math}C_1(x_1)=a^*\end{math}, a unique cover of \begin{math}a\end{math}. Apparently, \begin{math}x\in C_2(a)\end{math},  and therefore, \begin{math}x_1\in C_2(x)\end{math} implies \begin{math}x_1\in C_2(a)\end{math}. Now assume \begin{math}x_1=x_2\end{math} and denote \begin{math}b=C_2(x)\end{math}. If \begin{math}a=b\end{math}, then \begin{math}a, a\cup\{x_1\}\end{math} are in both chains, thus, \begin{math}a\end{math} is meet irreducible. So assume that \begin{math}a\not = b\end{math}. Then \begin{math}C_2(a) > b\end{math}, and therefore, \begin{math}x_1 \in C_2(a)\end{math}, which is needed.

      \item[\ref{ji5}]    Fix \begin{math}a\in \calG\end{math}   and let \begin{math}j\in X\end{math} be the greatest element of \begin{math}a\end{math} in  the ordering that defines chain \begin{math}C_i\end{math}. Then \begin{math}C_i(a)=C_i(j)\end{math} by the definition. Moreover, since \begin{math}j\in a\end{math}, we have \begin{math}(j)\leq a\end{math}.

      \item[\ref{ji10}.] Suppose \begin{math}x\in X\end{math} and without loss of generality, say \begin{math}C_1(p)\prec C_1(x)\end{math}. Fix \begin{math}u\in [\emptyset , C_1(x)]\end{math}. If \begin{math}u\geq (x)\end{math}, then \begin{math}u\in [(x),C_1(x)]\end{math}. So suppose \begin{math}u\ngeq (x)\end{math}. We know \begin{math}u=C_1(u)\wedge C_2(u)\end{math}.  Clearly \begin{math}C_1(u)\leq C_1(x)\end{math}.
      By (\ref{ji5}) above, since \begin{math}(x)\nleq u\end{math}, we know \begin{math}C_1(u)\neq C_1(x)\end{math}.  Thus, \begin{math}C_1(u)\leq C_1(p)\end{math}.  That is, \begin{math}u\in [\emptyset ,C_1(p)]\end{math}.
      In either case, \begin{math}u\in [\emptyset , C_1(p)]\cup [(x),C_1(x)]\end{math} and so  \begin{math}[\emptyset ,C_1(x)]\subseteq [\emptyset , C_1(p)]\cup [(x),C_1(x)]\end{math}.

      Now, suppose \begin{math}v\in [\emptyset , C_1(p)]\end{math}. Then \begin{math}v\leq C_1(p)\end{math} and so \begin{math}x\notin v\end{math} since \begin{math}x\notin C_1(p)\end{math}.  That is, \begin{math}(x)\nleq v\end{math} and so \begin{math}v\notin [(x),C_1(x)]\end{math}.  Thus, \begin{math}[\emptyset , C_1(p)]\end{math} and \begin{math}[(x),C_1(x)]\end{math} are disjoint.

\item[\ref{jiCi=X}.] Suppose without loss of generality that \begin{math}C_2(x)=X\end{math}. Then \begin{math}(x)=C_1(x)\wedge C_2(x)=C_1(x)\end{math}. If \begin{math}(x)\neq X\end{math} then there is \begin{math}x'\in X\end{math} such that \begin{math}C_1(x)\prec C_1(x')\leq X=C_2(x)\end{math}. Thus by part (\ref{C=M}) above, \begin{math}(x)=M_1(x)\end{math}.

\end{enumerate}
\end{proof}

\begin{lem}\label{theyarecomps} Let \begin{math}\calG =\left<  X, \calG \right>\end{math} be a convex geometry with  \begin{math}cdim=2\end{math} generated by chains \begin{math}C_1\end{math} and \begin{math}C_2\end{math} with trivial intersection and suppose \begin{math}C_i(j)\prec C_i(x_i)\end{math} for each \begin{math}i\in \{ 1,2\}\end{math}.
\begin{enumerate}
    \item\label{6.4_1} Suppose \begin{math}x_1=x_2\end{math}.

    \begin{enumerate}
    \item\label{6.4_1a} If \begin{math}(x_1)\end{math} is in neither chain, then \begin{math}(j)\end{math} is in neither chain, and \begin{math}[(j),C_1(j)]\cup [(j),C_2(j)]\end{math} is the complement of a sublattice of \begin{math}\calG \end{math}.

      \item\label{6.4_1b} If   \begin{math}(x_1)\in C_i\end{math} for one of the \begin{math}i\end{math}, then \begin{math}(j)=C_i(j)\end{math} is also in \begin{math}C_i\end{math}   and \begin{math}[(j),C_{i'}(j)]\end{math} is the complement of a sublattice of \begin{math}\calG \end{math}.

\end{enumerate}

        \item\label{6.4_2} If   \begin{math}x_1\neq x_2\end{math} and \begin{math}x_i\in C_{i'}(j)\end{math}, then   \begin{math}[(j), C_i(j)]\end{math} is the complement of a sublattice of \begin{math}\calG \end{math}.

        \item In all other cases \begin{math}[(j),C_{i'}(j)]\end{math} is not the complement of a sublattice.

 \end{enumerate}
\end{lem}

We note that if item \ref{6.4_1a} holds above, then necessarily \begin{math}C_1(j)=M_1(j)\end{math} and \begin{math}C_2(j)=M_2(j)\end{math} and if item   \ref{6.4_1b} or \ref{6.4_2} holds, then necessarily \begin{math}C_i(j)=M_i(j)\end{math} since the maximal elements of complements of sublattices are meet irreducible.

\begin{proof}  Assume the hypothesis.  We first show that a subset outside of the complement is closed under \begin{math}\vee\end{math}. Note that in all cases it is sufficient to show that if \begin{math}u,v\notin [(j), C_i(j)]\end{math}, then \begin{math}u\vee v\notin [(j),C_i(j)]\end{math}. Take \begin{math}u,v \in \calG\setminus [(j),C_i(j)]\end{math}. If \begin{math}u\not \leq C_i(j)\end{math}, then \begin{math}u\vee v \not \leq C_i(j)\end{math}.

 So now let \begin{math}u, v\leq C_i(j)\end{math} but \begin{math}(j)\not \leq u,v\end{math}.  If \begin{math}C_i(p)\prec C_i(j)\end{math}
then by Lemma \ref{JI} (\ref{ji10}) we know \begin{math}u,v\leq C_i(p)\end{math} and  hence \begin{math}u\vee v\leq C_i(p)\end{math} and thus is not in \begin{math}[(j),C_i(j)]\end{math}.  That is,   the complement of \begin{math}[(j),C_i(j)]\end{math} is closed under join for all \begin{math}j\in X\end{math}.  Thus to establish items \ref{6.4_1} and \ref{6.4_2} of the lemma we need only show the complement of the appropriate interval or union of intervals is also closed under meet and determine whether or not \begin{math}(j)\end{math} is on a  generating chain.

\begin{enumerate}

    \item     Suppose \begin{math}x_1=x_2=x\end{math}. Note that \begin{math}(j)=C_1(j)\wedge C_2(j)\leq C_1(x)\wedge C_2(x)=(x)\end{math} and so \begin{math}(j)<(x)\end{math}. But also, by semimodularity and since \begin{math}x\notin C_i(j)\end{math}, \begin{math}C_i(j)\wedge (x)\prec C_i(x)\wedge (x)= (x)\end{math}. But then \begin{math}(x)_*\in [(j), C_i(j)]\end{math} for each \begin{math}i\end{math} and so, by Lemma \ref{JI} (\ref{jiintersection}), \begin{math}(x)_*=(j)\end{math}.

    Fix \begin{math}w>(j)\end{math} such that \begin{math}w\notin [(j),C_1(j)]\cup [(j),C_2(j)]\end{math}.
    Then \begin{math} w =C_1(w)\wedge C_2(w)\end{math}. Since \begin{math} w \nleq C_i(j)\end{math}, \begin{math}C_i(w)>C_i(j)\end{math} and thus \begin{math}C_i(w)\geq C_i(x)\end{math} for each \begin{math}i\end{math}. But then \begin{math}x\in C_i(w)\end{math} for each \begin{math}i\end{math} and so \begin{math}x\in (w)\end{math}.  Thus, \begin{math}x\end{math} would be in the meet of any pair of elements greater than \begin{math}(j)\end{math} that are not in \begin{math}[(j),C_1(j)]\cup [(j),C_2(j)]\end{math} and so the complement of \begin{math}[(j),C_1(j)]\cup [(j),C_2(j)]\end{math} is   closed under the meet. Therefore, the complement of \begin{math}[(j),C_1(j)]\cup [(j),C_2(j)]\end{math} is a sublattice.

\begin{enumerate}
    \item Suppose that \begin{math}(x)\notin C_i\end{math}. Then \begin{math}(x)\neq C_i(x)\end{math} and so \begin{math}(x)< C_i(x)\end{math}. But then \begin{math}(j)<(x)<C_i(x)\end{math} and hence \begin{math}(j)\nprec C_i(x)\end{math}.  That is,  \begin{math}(j)\neq C_i(j)\end{math} and so \begin{math}(j)\end{math} is not on chain \begin{math}C_i\end{math}, either.

    Thus, if we assume \begin{math}(x)\end{math} is not in \begin{math}C_1\end{math} or \begin{math}C_2\end{math}, then   \begin{math}(j)\end{math} is in neither chain, as well.

    \item Without loss of generality, suppose \begin{math}(x)\in C_1\end{math}. Then \begin{math}(x)=C_1(x)\succ C_1(j)\geq (j)\end{math} and \begin{math}(x)\succ (j)\end{math}. Thus it must be the case that \begin{math}(j)=C_1(j)\end{math}; that is, that \begin{math}(j)\end{math} is on chain \begin{math}C_1\end{math}.

    We know that \begin{math}[(j), C_1(j)]\cup [(j), C_2(j)]\end{math} is the complement of a sublattice, but in this case \begin{math}[(j),C_1(j)]=\{ (j)\}\end{math} and so \begin{math}[(j), C_1(j)]\cup [(j), C_2(j)] =[(j), C_2(j)]\end{math} is the complement of a sublattice.

  \end{enumerate}

    \item Suppose \begin{math}x_1\neq x_2\end{math} and without loss of generality that \begin{math}x_1\in C_2(j)\end{math}. Then \begin{math}C_2(x_1)< C_2(j)\end{math}.

    Say \begin{math}u> (j)\end{math} and   \begin{math}u  \notin [(j),C_1(j)]\end{math}.  We know \begin{math}u=C_1(u)\wedge C_2(u)\end{math}. Since \begin{math}u\nleq C_1 (j)\end{math}, \begin{math}C_1(u)\geq  C_1(x_1) \succ C_1(j)\end{math}
    and hence \begin{math}x_1\in C_1(u)\end{math}. Also, since \begin{math}(j)<u\leq C_2(u)\end{math}, we know \begin{math}C_2(x_1)<C_2(j)\leq C_2(u)\end{math}. So, \begin{math}x_1\in C_2(u)\end{math} as well. That is, \begin{math}x_1\in C_1(u)\wedge C_2(u)=u\end{math}.   Hence, the meet of any two elements bigger than \begin{math}(j)\end{math} but not in \begin{math}[(j),C_1(j)]\end{math} will contain \begin{math}x_1\end{math} and, hence, will also not be in \begin{math}[(j),C_1(j)]\end{math}.   Therefore, the complement of \begin{math}[(j),C_1(j)]\end{math} is a sublattice.

    \item
    Suppose \begin{math}x_1\neq x_2\end{math} and, without loss of generality, that \begin{math}x_1\notin C_2(j)\end{math}. We need to show that \begin{math}[(j), C_1(j)]\end{math} is not the complement of a sublattice.

    Since \begin{math}x_1\notin C_2(j)\end{math}, \begin{math}C_2(j)<C_2(x_1)\end{math} and since \begin{math}x_1\neq x_2\end{math} and \begin{math}C_2(j)\prec C_2(x_2)\end{math}, this implies \begin{math}C_2(x_2)<C_2(x_1)\end{math}. Then \begin{math}x_1\notin C_2(x_2)\end{math} and hence \begin{math}x_1\notin  (x_1)\cap C_2(x_2)\end{math}. So, \begin{math}(x_1)\wedge C_2(x_2)<(x_1)\end{math}. However, it is the case that \begin{math}j\in C_2(x_1)\end{math} and so \begin{math}j\in C_1(x_1)\cap C_2(x_1)=\{ (x_1)\} \end{math} and \begin{math}j\in C_2(x_2)\end{math}.  That is, \begin{math}(j)\leq (x_1)\wedge C_2(x_2)<(x_1)\end{math}.

    By lower semi-modularity, since \begin{math}C_1(j)\prec C_1(x_1)\end{math}, either \begin{math}C_1(j)\wedge (x_1)=(x_1)\end{math} or \begin{math}C_1(j)\wedge (x_1)\prec (x_1)\end{math}.  Since \begin{math}C_1(j)\ngeq (x_1)\end{math}, it must be the case that \begin{math}C_1(j)\wedge (x_1)=(x_1)_* \prec (x_1)\end{math}.  But then, since \begin{math}(x_1)\wedge C_2(x_2)<(x_1)\end{math},
    we obtain \begin{math}(j)\leq (x_1)\wedge C_2(x_2)\leq (x_1)_*=C_1(j)\wedge (x_1)\leq C_1(j)\end{math}. That is, \begin{math}(x_1)\wedge C_2(x_2)\in [(j),C_1(j)]\end{math}, while
    \begin{math}(x_1), C_2(x_2) \not \leq C_1(j)\end{math}.
    Therefore, \begin{math}[(j),C_1(j)]\end{math} is not the complement of a sublattice.
\end{enumerate}
\end{proof}

\begin{lem}\label{allfoursect6lemmas}  Suppose \begin{math}\calG = \left< X, \calG  \right>\end{math} is a convex geometry of \begin{math}cdim =2\end{math} generated by chains \begin{math}C_1\end{math} and \begin{math}C_2\end{math}, with trivial intersection, and that \begin{math}\calC\end{math} is the complement of a sublattice \begin{math}\calS\end{math} of \begin{math}\calG \end{math}. Let \begin{math}(j)\in \calC\end{math} be a maximal join irreducible element of \begin{math}\calC\end{math}.

\begin{enumerate}
    \item\label{maxji1} Suppose that for some \begin{math}x\in X\end{math}, \begin{math}C_i(j)\prec C_i(x)\end{math} for each \begin{math}i=1,2\end{math}. Then \begin{math}[(j),C_1(j)]\cup [(j),C_2(j)]\subseteq \calC\end{math}.

\item\label{nolabel} If \begin{math}C_i(j)\prec C_i(x_i)\end{math} for some \begin{math}x_1\neq x_2\end{math}, then either \begin{math}x_1\in C_2(j)\end{math} or \begin{math}x_2\in C_1(j)\end{math}.

\item\label{maxji2}  Suppose that \begin{math}C_i(j)\prec C_i(x_i)\end{math} for some \begin{math}x_1\neq x_2\end{math}. Then for some \begin{math}i\in \{ 1,2\}\end{math}, \begin{math}[(j),C_i(j)]\subseteq \calC\end{math} and \begin{math}C_{i'}(x_i)<C_{i'}(j)\end{math}.

\item\label{maxji3} If \begin{math}C_i(j)\prec C_i(x_i)\end{math} for some \begin{math}x_1\neq x_2\end{math}, then for at least one of \begin{math}i\in\{ 1,2\}\end{math} we have \begin{math}C_{i'}(x_i)<C_{i'}(j)\end{math}, \begin{math}[(j),C_i(j)]\end{math} is a subset of \begin{math}\calC\end{math},   and \begin{math}[(j),C_i(j)]\end{math} is itself the complement of a  sublattice.
\end{enumerate}
    \end{lem}

\begin{proof} Assume the hypothesis.

\begin{enumerate}
    \item[(\ref{maxji1})] Suppose that for some \begin{math}x\in X\end{math}, \begin{math}C_i(j)\prec C_i(x)\end{math} for each \begin{math}i=1,2\end{math}.
We know \begin{math}(j) =C_1(j)\wedge C_2(j)\leq C_1(x)\wedge C_2(x)=(x)\end{math} and since \begin{math}(j)\neq (x)\end{math}, this implies \begin{math}(j)<(x)\end{math} and hence by our choice of \begin{math}(j)\end{math}, \begin{math}(x)\in\calS\end{math}. Also, by semimodularity, since \begin{math}x\notin C_i(j)\end{math} but \begin{math}x\in (x)\end{math}, \begin{math}C_i(j)\wedge (x)\prec (x)\end{math}. But, \begin{math}j\in C_i(j)\wedge (x)=(x)_*\end{math} and so \begin{math}(x)_*\in [(j),C_i(j)]\end{math} for each \begin{math}i\end{math}, which means that \begin{math}(j)=(x)_*\end{math} by Lemma \ref{JI} (\ref{jiintersection}). Furthermore, since \begin{math}(j)=(x)_*\end{math} and \begin{math}(x)\notin [(j), C_i(j)]\end{math} for either \begin{math}i\end{math}, for each \begin{math}b\in [(j), C_i(j)]\end{math}, we know \begin{math}b\wedge (x)=(j)\end{math} and so \begin{math}b\in\calC\end{math}. That is,   \begin{math}[(j),C_1(j)]\cup [(j),C_2(j)]\subseteq\calC\end{math}.

\item[(\ref{nolabel})] Suppose that \begin{math}C_i(j)\prec C_i(x_i)\end{math} for some \begin{math}x_1\neq x_2\end{math}. Suppose by way of contradiction that \begin{math}x_1\not \in C_2(j)\end{math} and \begin{math}x_2\not \in C_1(j)\end{math}. Then  \begin{math}C_2(j) < C_2(x_1)\end{math}. Therefore, \begin{math}(x_1)=C_1(x_1)\cap C_2(x_1) > C_1(j)\cap C_2(j)=(j)\end{math} and so by our choice of \begin{math}j\end{math}, \begin{math}(x_1)\in\calS\end{math}.
   Since \begin{math}C_1(j)\prec C_1(x_1)\end{math}, by lower semimodularity we know \begin{math}C_1(j)\wedge C_2(x_1)\prec C_1(x_1)\wedge C_2(x_1)=(x_1)\end{math}. That is, \begin{math}C_1(j)\wedge C_2(x_1)=(x_1)_*\end{math}. But we know then, since \begin{math}(j)<(x_1)\end{math}, \begin{math}(j)\leq (x_1)_*\leq C_1(j)\end{math}; that is, \begin{math}(x_1)_*\in [(j),C_1(j)]\end{math}.  Analogously, \begin{math}(x_2)\in\calS\end{math} and \begin{math}(x_2)_*\in [(j),C_2(j)]\end{math} and \begin{math}(x_2)\in \calS\end{math}.  But then, since \begin{math}x_1\neq x_2\end{math} and elements \begin{math}(x_1)\end{math} and \begin{math}(x_2)\end{math} are incomparable, thus    \begin{math}(x_1)\wedge (x_2)=(x_1)_*\wedge (x_2)_*=(j)\end{math}, which is a contradiction since \begin{math}(x_1),(x_2)\in \calS\end{math}.

     Thus, either \begin{math}x_1\in C_2(j)\end{math} or \begin{math}x_2\in C_1(j)\end{math}.

\item[(\ref{maxji2})] Suppose \begin{math}C_i(j)\prec C_i(x_i)\end{math} and \begin{math}x_1\neq x_2\end{math}. Assume without loss of generality that \begin{math}x_1\in C_2(j)\end{math}.  By Lemma 6.3(1), we know for each \begin{math}c_1\in [(j),C_1(j)]\end{math} and \begin{math}c_2\in [(j),C_2(j)]\end{math}, \begin{math}(j)=c_1\wedge c_2\end{math} and so at least one of \begin{math}[(j),C_i(j)]\end{math} is a subset of \begin{math}\calC\end{math}. If \begin{math}[(j),C_1(j)]\subseteq \calC\end{math}, then we are done. Suppose instead that \begin{math}[(j),C_1(j)]\nsubseteq \calC\end{math}. Then \begin{math}[(j),C_2(j)]\subseteq \calC\end{math}.  We need to show in this case that \begin{math}x_2\in C_1(j)\end{math}.
Suppose by way of contradiction that  \begin{math}x_2\notin C_1(j)\end{math}.  Then, \begin{math}C_1(j)\prec C_1(x_1)<C_1(x_2)\end{math} and \begin{math}C_2(x_1)<C_2(j)\prec C_2(x_2)\end{math}. We know, \begin{math}(x_2)=C_1(x_2)\wedge C_2(x_2)\geq C_1(j)\wedge C_2(j)=(j)\end{math} and \begin{math}(x_2)\neq (j)\end{math}. Hence, by our choice of \begin{math}(j)\end{math}, \begin{math}(x_2)\in \calS\end{math}. By lower semimodularity, since \begin{math}C_2(j)\prec C_2(x_2)\end{math} and \begin{math}x_2\notin C_2(j)\end{math}, we know \begin{math} C_1(x_2)\wedge C_2(j)\prec C_1(x_2)\wedge C_2(x_2)=(x_2)\end{math}. But \begin{math}(j)\leq C_1(x_2)\wedge C_2(j)\leq C_2(j)\end{math}. That is, \begin{math}(x_2)_*\in [(j), C_2(j)]\end{math}. Since \begin{math}[(j),C_1(j)]\nsubseteq \calC\end{math}, there is \begin{math}c\in [(j),C_1(j)]\cap \calS\end{math} and \begin{math}c\wedge (x_2)=c\wedge (x_2)_*=(j)\end{math}, contradicting the fact that \begin{math}(j)\in \calC\end{math} since both \begin{math}c,(x_2)\in\calS\end{math}.

\item[(\ref{maxji3})]  Suppose \begin{math}C_i(j)\prec C_i(x_i)\end{math} and \begin{math}x_1\neq x_2\end{math}.  By Lemma \ref{allfoursect6lemmas} (\ref{nolabel}) either \begin{math}x_1\in C_2(j)\end{math} or \begin{math}x_2\in C_1(j)\end{math}. Without loss of generality, suppose that \begin{math}x_1\in C_2(j)\end{math}. Then, by Lemma \ref{theyarecomps} (\ref{6.4_2}), \begin{math}[(j),C_1(j)]\end{math} is the complement of a sublattice. By Lemma \ref{allfoursect6lemmas} (\ref{maxji2}) either \begin{math}[(j),C_1(j)]\subseteq \calS\end{math} or also \begin{math}x_2\in C_1(j)\end{math} and \begin{math}[(j),C_2(j)]\subseteq \calS\end{math} and again by Lemma \ref{theyarecomps} (\ref{6.4_1b}),   \begin{math}[(j),C_2(j)]\end{math} is the complement of a sublattice.

Note that in each case, there is an \begin{math}i\in\{ 1,2\}\end{math} such that \begin{math}C_{i'}(x_i)<C_{i'}(j)\end{math}  and \begin{math}[(j),C_i(j)]\subseteq \calC\end{math}.

\end{enumerate}
\end{proof}

We finally prove Theorem \ref{BigCGThm}.

\begin{proof} (of Theorem \ref{BigCGThm})
Let \begin{math}\calG =\left<  X, \calG \right>\end{math} be a convex geometry with \begin{math}cdim=2\end{math} and generating chains \begin{math}C_1\end{math} and \begin{math}C_2\end{math}.  In light of Observation \ref{obs0}  from Section \ref{observations} we may assume \begin{math}C_1\end{math} ad \begin{math}C_2\end{math} have trivial intersection. Let \begin{math}\calC\end{math} be the complement of some maximal sublattice \begin{math}\calM\end{math}. \begin{math}\calC\end{math} has at least one minimal element, which must be join irreducible. Thus, we can choose a maximal join irreducible element \begin{math}(j)\end{math} in \begin{math}\calC\end{math}.

\begin{enumerate}

\item Suppose that, without loss of generality \begin{math}C_2(j)=X\end{math}. Then by  Lemma \ref{JI} (\ref{jiCi=X}), \begin{math}(j)=M_1(j)\end{math}. That is \begin{math}(j)\end{math} is doubly irreducible  and hence \begin{math}\calC =\{ (j)\}\end{math}. We know \begin{math}(j)\neq X\end{math} since \begin{math}X\in \calM\end{math} and so there is an \begin{math}x\in X\end{math} with \begin{math}(j)=C_1(j)\prec C_1(x)\end{math}. Since \begin{math}C_1(j)\end{math} is meet irreducible, \begin{math}(j)=C_1(j)\prec C_1(x)\leq C_2(j)\end{math}.
That is, \begin{math}[(j), C_1(j)]=\{ (j)\}\end{math}  is a complement of type (\ref{cg1}) of Theorem \ref{BigCGThm}.

\item Suppose that \begin{math}C_i(j)\neq X\end{math} for either \begin{math}i\end{math}. Then for some \begin{math}x_i\in X\end{math},  \begin{math}C_i(j)\prec C_1(x_i)\end{math}.  We have two cases: \begin{math}x_1\neq x_2\end{math} or \begin{math}x_1=x_2\end{math}.

\begin{enumerate}

\item Suppose \begin{math}x_1\neq x_2\end{math}.   Then by Lemma \ref{allfoursect6lemmas} (\ref{maxji3}) for some \begin{math}i\in \{ 1, 2\}\end{math} we have \begin{math}C_{i'}(x_i)<C_{i'}(j)\end{math}, \begin{math}[(j),C_i(j)]\subseteq \calC\end{math} and  \begin{math}[(j),C_i(j)]\end{math} is the complement of a sublattice. Since \begin{math}\calM\end{math} is a maximal sublattice, it must be the case that \begin{math}[(j),C_i(j)]=\calC\end{math}.

Notice that in this case,  since \begin{math}x_i\notin C_i(j)\end{math}, we know \begin{math}x_i\notin (j)\end{math} and \begin{math}x_i\in C_{i'}(j)\end{math} and thus \begin{math}(j)\neq C_{i'}(j)\end{math}. Hence, \begin{math}(j)\end{math} is not on chain \begin{math}C_{i'}\end{math}.
Thus in this case \begin{math}\calC\end{math} is a complement of type (\ref{cg1}) of Theorem \ref{BigCGThm}.

\item Suppose \begin{math}x_1=x_2\end{math}. Then by Lemma \ref{allfoursect6lemmas} (\ref{maxji1}) we have one of the two following possibilities:
\begin{enumerate}

 \item \begin{math}(x_1)\end{math} is on neither chain, in which case neither is \begin{math}(j)\end{math} and \begin{math}[(j),C_1(j)]\cup [(j),C_2(j)]\subseteq \calC\end{math} and is itself the complement of a sublattice, by Lemma (\ref{theyarecomps})(\ref{6.4_1})(\ref{6.4_1a}).   Thus, since \begin{math}\calM\end{math} is maximal, \begin{math}\calC=[(j),C_1(j)]\cup [(j),C_2(j)]\end{math}.   In this case \begin{math}\calC\end{math} is a complement of type (\ref{cg2}) of Theorem \ref{BigCGThm}.

 \item \begin{math}(x_1)\end{math} is on \begin{math}C_i\end{math} for one \begin{math}i\end{math}, in which case \begin{math}(j)\end{math} is also on \begin{math}C_i\end{math} and \begin{math}[(j),C_{i'}(j)]\subseteq \calC\end{math} is the complement of a sublattice by Lemma (\ref{theyarecomps})(\ref{6.4_1})(\ref{6.4_1b}).    But then \begin{math}\calC=[(j),C_{i'}(j)]\end{math} since \begin{math}\calM\end{math} is maximal. Thus \begin{math}\calC\end{math} is a complement of type (\ref{cg3}) of Theorem \ref{BigCGThm}.
\end{enumerate}
\end{enumerate}
\end{enumerate}
\end{proof}

Note that Hypotheses \ref{hyp2} and \ref{hyp3}   for CGs with \begin{math}cdim=2\end{math} follow from Theorem \ref{BigCGThm} immediately, as does an affirmative answer for Question \ref{ques2}.  In fact, we also have Corollary \ref{hyp4forCGcdim2}, below.

\begin{cor}\label{hyp4forCGcdim2} For every element \begin{math}a\end{math} in the complement of a maximal sublattice \begin{math}\calM\end{math} of a CG with \begin{math}cdim=2\end{math}, there is an \begin{math}m\in\calM\end{math} with \begin{math}m\prec a\end{math}.

\end{cor}

\begin{proof} We just have to note that the only join irreducible element of an interval \begin{math}[(j),C_i(j)]\end{math} is \begin{math}(j)\end{math}.  Indeed, if \begin{math}(j)\leq (s)\leq C_i(j)\end{math}, then \begin{math}C_i(s)=C_i(j)\end{math} as \begin{math}C_i(j)\end{math} must necessarily be the first element of \begin{math}C_i\end{math} bigger than or equal to \begin{math}(s)\end{math}. Thus \begin{math}s=j\end{math}.
\end{proof}

This suggests the following hypothesis in convex geometries.

\begin{hypothesis}\label{hyp4}
For every element \begin{math}x\in \calC\end{math} of convex geometry \begin{math}(X,\calG)\end{math}, where \begin{math}\calC\end{math} is the complement of a maximal sublattice \begin{math}\calM\end{math}, there exists \begin{math}m\prec x\end{math} with \begin{math}m\in \calM\end{math}. Moreover, \begin{math}[0,m]\subseteq \calM\end{math}.
\end{hypothesis}

Lemma  \ref{hyp4implieshyp3} tells us if Hypothesis \ref{hyp4} is correct, then Hypothesis \ref{hyp3} is correct for CGs.

\begin{lem}\label{hyp4implieshyp3} If Hypothesis \ref{hyp4} is correct, then minimal complement \begin{math}\calC\end{math} must be convex.
\end{lem}
\begin{proof}
Suppose \begin{math}c_1< m< c_2\end{math} with \begin{math}m\in \calM\end{math}, \begin{math}c_1,c_2\in C\end{math}. By the hypothesis, there exists \begin{math}m_1\prec c_2\end{math}, \begin{math}m_1\in \calM\end{math}. Then \begin{math}m\leq m_1\end{math}, otherwise, \begin{math}m\vee m_1=c_2 \in \calM\end{math}. But then \begin{math}c_1 \in [0,m_1]\end{math}, a contradiction.
\end{proof}

 Note that the hypothesis fails in bounded lattices, as Figure \ref{daydoubling2} demonstrates.

\subsection{Algorithm}\label{algorithm}
\def\krok#1#2{#1 & #2 \\}
\def\io#1#2{{\begin{math}\ \end{math}}\\ \begin{tabular}{ll} {\bf In:} & #1 \\ {\bf Out:} & #2 \end{tabular} \\}
\def\afont#1{{\bf #1}}
Theorem \ref{BigCGThm} leads us to the procedure of finding all complements of maximal sublattices of a convex geometry \begin{math}\calG =\left<  X, \calG \right>\end{math} with \begin{math}cdim=2\end{math} generated by chains \begin{math}C_1\end{math} and \begin{math}C_2\end{math}.  For simplicity, we exploit the chain representation assuming that the first chain is always of the form:
\begin{equation*}
1 <_1 2 <_1 \dots <_1 n.
\end{equation*}
Then, the input to the algorithm is only number \begin{math}n\end{math} of elements and some permutation mapping \begin{math}\phi\end{math} such that the second chain is of the form:
\begin{equation*}
\phi(1) <_2 \phi(2) <_2 \dots <_2 \phi(n).
\end{equation*}

\begin{algorithm}\label{alg:complements}
\noindent
\io{\begin{math}n=|X|\end{math} and a permutation \begin{math}\phi\colon X\to X\end{math}}{all complements of the maximal sublattices in \begin{math}\left<  X, \calG \right>\end{math}.}
\begin{tabular}{ll}
\krok{1.}{\begin{math}j:=1\end{math}}
\krok{2.}{\afont{while} \begin{math}j \leq n\end{math}}
\krok{3.}{\quad \begin{math}C_1(j):=\{1,2,\ldots ,j\}\end{math}}
\krok{4.}{\quad \begin{math}C_2(j):=\{\phi(1),\ldots ,\phi(\phi^{-1}(j))\}\end{math}}
\krok{5.}{\quad \begin{math}(j):=C_1(j)\cap C_2(j)\end{math}}
\krok{6.}{\quad \begin{math}j:=j+1\end{math}}
\krok{7.}{\afont{endwhile}}
\krok{8.}{\begin{math}j:=1\end{math}}
\krok{9.}{\afont{while} \begin{math}j\leq n\end{math}}
\krok{10.}{\quad\afont{if} \begin{math}\phi^{-1}(j)\neq n\end{math} \afont{and} \begin{math}j+1=\phi(\phi^{-1}(j)+1)\end{math}}
\krok{11.}{\qquad\afont{if} \begin{math}(j)=C_2(j)\end{math}}
\krok{12.}{\quad\qquad\afont{print} \begin{math}[(j),C_1(j)]\end{math}}
\krok{13.}{\qquad\afont{else}}
\krok{14.}{\quad\qquad\afont{if} \begin{math}(j)=C_1(j)\end{math}}
\krok{15.}{\qquad\qquad\afont{print} \begin{math}[(j),C_2(j)]\end{math}}
\krok{16.}{\quad\qquad\afont{else}}
\krok{17.}{\qquad\qquad\afont{print} \begin{math}[(j),C_1(j)]\cup[(j),C_2(j)]\end{math}}
\krok{18.}{\quad\qquad\afont{endif}}
\krok{19.}{\qquad\afont{endif}}
\krok{20.}{\quad\afont{else}}
\krok{21.}{\qquad\afont{if} \begin{math}j+1\in C_2(j)\end{math}}
\krok{22.}{\quad\qquad\afont{print} \begin{math}[(j),C_1(j)]\end{math}}
\krok{23.}{\qquad\afont{endif}}
\krok{24.}{\qquad\afont{if} \begin{math}\phi^{-1}(j)\neq n\end{math} \afont{and} \begin{math}\phi(\phi^{-1}(j)+1)\in C_1(j)\end{math}}
\krok{25.}{\quad\qquad\afont{print} \begin{math}[(j),C_2(j)]\end{math}}
\krok{26.}{\qquad\afont{endif}}
\krok{27.}{\quad\afont{endif}}
\krok{28.}{\quad \begin{math}j:=j+1\end{math}}
\krok{29.}{\afont{endwhile}}
\end{tabular}
\end{algorithm}

\begin{exam}
For the CG presented on Figure \ref{fig:CompsOfCG2}, using the procedure described in Algorithm \ref{alg:complements},
we can find all complements of maximal sublattices, and, obviously, at the same time all maximal sublattices. In Table \ref{T:Example} the values for loops conditional statements are collected.
\begin{table}[hbt]
  \begin{center}
\begin{small}
\begin{equation*}
\begin{array}{|r|rrrrrrrrrr|}\hline
j                       & 1& 2& 3& 4&  5&  6&   7&    8&     9&    10\\\hline
\phi(j)           & 3& 6& 7& 10& 1& 8& 9& 5& 2& 4\\
\phi^{-1}(j)           & 5& 9& 1& 10& 8& 2& 3& 6& 7& 4\\
\phi(\phi^{-1}(j)+1)    & 8& 4& 6& -&  2&  7&   10&    9&    5&     1\\ \hline
\end{array}
\end{equation*}
\end{small}
 \caption{Values for the loop's conditional statements. \begin{math}4\end{math} is the greatest element in \begin{math}C_2\end{math}, so it has no successor.}
\label{T:Example}
  \end{center}
\end{table}
We obtain as complements: \begin{math}\{(2)=C_1(2)\}\end{math}, \begin{math}\{(4)=C_1(4)\}\end{math}, \begin{math}[(5),C_1(5)]\end{math}, \begin{math}[(5),C_2(5)]\end{math}, \begin{math}[(6),C_1(6)]\end{math}, \begin{math}[(8),C_1(8)]\cup[(8),C_2(8)]\end{math}, \begin{math}[(9),C_1(9)]\end{math}, \begin{math}[(9),C_2(9)]\end{math}, \begin{math}\{(10)=C_2(10)\}\end{math}. The singletons are just doubly irreducible elements.

\end{exam}

\begin{prop}\label{prop:complexity}
Algorithm \ref{alg:complements} runs in \begin{math}\mathcal{O}(n)\end{math} time with respect to \begin{math}n=|X|\end{math}.
\end{prop}
\begin{proof} Notice that the representation of the input to the algorithm, as presented in \ref{algorithm},  is of the size \begin{math}n\end{math}. We would like to answer the question regarding complexity relating to the size of the input. In the loop, the algorithm processes each element of \begin{math}X\end{math}. For each element \begin{math}j\end{math}, where \begin{math}j=1,2,\dots n\end{math}, we need to compute the following objects:

\begin{itemize}
    \item \begin{math}C_1(j) = \{1,2,\dots,n\}\end{math}
    \item \begin{math}C_2(j) = \{\phi(1),\phi(2),\dots,\phi(\phi^{-1}(j))\}\end{math}
    \item \begin{math}(j) = C_1(j) \cap C_2(j)\end{math}.
\end{itemize}

All of this  may be computed at the preprocessing stage of the algorithm implementation, going from the bottom element of the chain up to the top one. Each of these computations has complexity   \begin{math}\mathcal{O}(n)\end{math} for all elements in \begin{math}X\end{math}. Hence the cost of preprocessing is \begin{math}\mathcal{O}(n)\end{math}.

Now, let us list the types of comparisons which are made within the algorithm, for some \begin{math}j, k \in \{1,2,\dots,n\}\end{math}:

\begin{enumerate}
    \item Does \begin{math}(j) = C_1(j)\end{math} ?
    \item Does \begin{math}(j) = C_2(j)\end{math} ?
    \item Does \begin{math}j + 1 \in C_2(j)\end{math} ?
    \item Does \begin{math}\phi^{-1}(j) \not = n\end{math} ?
    \item Does \begin{math}j+1=\phi(\phi^{-1}(j)+1)\end{math} ?
    \item Does \begin{math}\phi(\phi^{-1}(j)+1) \in C_1(j)\end{math} ?
\end{enumerate}

The complexity of each step is \begin{math}\mathcal{O}(1)\end{math} assuming that the preprocessing has been carried out. Since every step is of \begin{math}\mathcal{O}(1)\end{math} time complexity and these steps are performed for each \begin{math}a \in X\end{math}, the complexity of the algorithm main part is \begin{math}\mathcal{O}(n)\end{math}.

Having the preprocessing at a time cost of \begin{math}\mathcal{O}(n)\end{math} and the main part of the algorithm with the same complexity, we find that the time complexity of the presented algorithm is \begin{math}\mathcal{O}(n)\end{math}.

\end{proof}

We present the running time of the algorithm in the table. The computation was made using Sage Math Cloud, CoCalc, with a base CPU with \begin{math}3.0\end{math} GHz clock speed and a booster CPU adding up to \begin{math}4.0\end{math} GHz clock speed on the underlying CPU platform (Xeon or EPYC, depending on the run).

\begin{table}[hbt]
  \begin{center}
\begin{small}
\begin{equation*}
\begin{array}{|r|l|l|}\hline
n & \text{\textbf{Algorithm 6.9} computation time} & \text{CoCalc native method computation time}\\\hline
10 & \sim 0.77\ seconds & \sim 21.16\ seconds\\
20 & \sim 0.82\ seconds & \sim 605.36\ seconds\\
30 & \sim 0.96\ seconds & Did\ not\ accomplish\\
50 & \sim 1.02\ seconds & Did\ not\ accomplish\\
70 & \sim 20.81\ seconds & Did\ not\ accomplish\\
90 & \sim 57.96\ seconds & Did\ not\ accomplish\\
100 & \sim 79.01\ seconds & Did\ not\ accomplish\\
\hline
\end{array}
\end{equation*}
\end{small}
 \caption{Average computation time for different \begin{math}n\end{math}.}
\label{T:ComputationTimes}
  \end{center}
\end{table}

The native method available in CoCalc Sage Math Cloud is \begin{math}maximal\_sublattices()\end{math}. The input to the Algorithm \ref{alg:complements} is a permutation  \begin{math}\phi\end{math} which permutes elements \begin{math}\{1,\dots,n\}\end{math} and hence provides the second chain creating a convex geometry:
\begin{equation*}
\phi(1) < \phi(2) < \dots < \phi(n)
\end{equation*}
assuming that the first generating chain is:
\begin{equation*}
1 < 2 < \dots < n.
\end{equation*}
The algorithm exploiting the CoCalc method \begin{math}maximal\_sublattices()\end{math} requires the input in the form of the collection of pairs \begin{math}(s, e)\end{math}, sitting within a convex geometry where \begin{math}s\end{math} is a subcover of \begin{math}e\end{math}. The translation from \begin{math}\phi\end{math} representation to a subcover representation is provided in the preprocessing stage of the computation time comparison, although the translation time is not involved in the time measurement.

\subsection{Future Work}\label{futurework}

At the time of submission of this work, we have preliminary results on the structure of the Frattini sublattice (the intersection of maximal sublattices of a lattice) obtained for convex geometries of \begin{math}cdim=2\end{math}. They are based on Theorem \ref{BigCGThm}. For example, we have shown that the Boolean lattice on a 2-element set cannot be the Frattini sublattice of a CG with \begin{math}cdim=2\end{math}, which has already been shown  in \cite{A73} and \cite{CKT75} not to be the Frattini of any distributive lattice. On the other hand, in \cite{AFNS97} it has been shown that every finite bounded lattice can be represented as a Frattini lattice for a finite bounded lattice.

We also have initial observations about extending the approach and argument of Theorem \ref{BigCGThm} to convex geometries of \begin{math}cdim=3\end{math}, but they ensure more variability in cases one needs to explore.

\acknowledgements\label{sec:ack}
We are grateful to Jonathan Farley who brought our attention to Question \ref{ques1} for SD lattices, during his visit to Hofstra University in March 2018. Several undergraduate Hofstra students (Genevieve Maalouf and Sadman Sakib among them) honed their research skills by joining the seminars exploring this problem. The last author is grateful for the Fulbright Senior Award granted by the Polish-U.S. Fulbright Commission supporting her stay in the US during the 2021-2022 academic year and encouraging the participation in this project. The second author thanks Warsaw University of Technology for providing the grant for visiting Hofstra University on this project in the Fall of 2023 within the SEED - Smart Education for Engineering Doctors project under STER - Internationalization of doctoral schools programme financed by the Polish National Agency for Academic Exchange (Agreement no. PPI/STE/2020/1/00018/U/00001 dated 21.12.2020). The third author expresses gratitude to Hofstra University for supporting the research trip to Europe in May-June 2025, when the results of this project were presented.

\label{sec:biblio}


\begin{thebibliography}{99}
\bibitem[Adams(1973)]{A73} M. E. Adams, \textit{The Frattini sublattice of a distributive lattice}, Algebra Universalis {\bf 3} (1973), 216--228.

\bibitem[Adams et al.(1997)]{AFNS97} M. E. Adams, R. Freese, J. B. Nation, J. Schmid, \textit{Maximal Sublattices and Frattini Sublattices of Bounded Lattices}, J. Austral. Math. Soc., A {\bf 63} (1997), 110--127.

\bibitem[Adaricheva and Nation(2016)]{AN16} K. Adaricheva, J.B. Nation, \textit{Convex geometries}, in Lattice Theory: Special Topics and Applications, v.2, G.~Gr\"atzer and F.~Wehrung, eds. Springer, Basel, 2016.

\bibitem[Adaricheva(2016)]{A16} K. Adaricheva, private communication with J.B. Nation and J. M\"antisalo.

\bibitem[Adaricheva et al.(2024)]{5circles} K. Adaricheva et al., \textit{Convex geometries representable by at most five circles on the plane}, Involve, a Journal of Mathematics {\bf 17} (2024), 337--354. Appendices (433 pages) to be found at \texttt{https://arxiv.org/abs/2008.13077}

\bibitem[Cz\'edli and Kurusa(2019)]{CK19} G. Cz\'edli, \'A. Kurusa, \textit{A convex combinatorial property of compact
sets in the plane and its roots in lattice theory}, Categ. Gen. Algebr. Struct. Appl. {\bf 11} (2019), 57--92.

\bibitem[Chen et al.(1973)]{Chen73} C.C. Chen, K.M. Koh, S.K. Tan, \textit{Frattini sublattices of distributive lattices}, Algebra Universalis {\bf 3} (1973), 294--303.

\bibitem[Chen et al.(1975)]{CKT75}  C.C. Chen, K.M. Koh, S.K. Tan, \textit{On the Frattini sublattice of a finite distributive lattice}, Algebra Universalis {\bf 5} (1975), 88--97.

\bibitem[Day(1992)]{Day1992} A. Day, \textit{Doubling construction in lattice theory}, Canad. J. Math. {\bf 44} (1992), 252--269.

\bibitem[Edelman and Jamison(1985)]{EJ85} P. H. Edelman, R. E. Jamison, \textit{The theory of convex geometries}, Geom. Dedicata {\bf 19} (1985), 247--270.

\bibitem[Freese et al.(1995)]{FJN95} R. Freese, J. Je\v{z}ek, J.B. Nation, \textit{Free Lattices}, AMS, 1995.

\bibitem[Ganter(2019)]{Gan19} B. Ganter, \textit{``Properties of finite lattices'' by S. Reeg and W. Weiss, revisited}, ICFCA 2019, LNAI 11511, D. Cristea et al. (Eds), 99--109.

\bibitem[Gr\"atzer(2023)]{G2}   G. Gr\"atzer, \textit{The Congruences of a Finite Lattice 3rd Edition; A ``Proof-by-Picture'' Approach}, Birkh\"auser, 2023.

\bibitem[Rival(1973)]{R73} I. Rival,  \textit{Maximal Sublattices of Finite Distributive Lattices},  Proc. Amer. Math. Soc. {\bf 37} (1973), 417--420.

\bibitem[Schmid(1999)]{Sch99}
J. Schmid, \textit{On maximal sublattices of finite lattices}, Discrete Math. {\bf 199} (1999), 151--159.
\bibitem[Stern(1999)]{S99}
M. Stern, \textit{Semimodular lattices: theory and applications}, Cambridge University Press, 1999.

\end{thebibliography}
\end{document}